\documentclass[a4paper,12pt]{article}

\usepackage{amscd}
\usepackage{amsfonts}
\usepackage{amsmath}
\usepackage{amssymb}
\usepackage{amsthm}
\usepackage[T1]{fontenc} 
\usepackage{here} 
\usepackage{mathrsfs} 
\usepackage{txfonts} 
\usepackage[all]{xy} 
\usepackage{algorithm} 
\usepackage{algpseudocode} 
\usepackage{diagbox} 

\allowdisplaybreaks

\theoremstyle{plain}
\newtheorem{thm}{Theorem}[section]
\newtheorem{lmm}[thm]{Lemma}
\newtheorem{prp}[thm]{Proposition}
\newtheorem{crl}[thm]{Corollary}

\theoremstyle{definition}
\newtheorem{dfn}[thm]{Definition}

\newcommand{\vs}[1][0.2]{\vspace{#1in}\noindent\ignorespaces}
\newcommand{\ba}{\begin{array*}}
\newcommand{\ea}{\end{array*}}
\newcommand{\be}{\begin{eqnarray*}}
\newcommand{\ee}{\end{eqnarray*}}
\newcommand{\bi}{\begin{itemize}}
\newcommand{\ei}{\end{itemize}}
\newcommand{\bb}{\vs\begin{itembox}}
\newcommand{\eb}{\end{itembox}}
\newcommand{\bc}{\begin{center}}
\newcommand{\ec}{\end{center}}
\newcommand{\bs}{\vs\begin{screen}}
\newcommand{\es}{\end{screen}}

\def\ens#1{{\mathchoice{\left\{ #1 \right\}}{\{ #1 \}}{\{ #1 \}}{\{ #1 \}}}}
\def\set#1#2{{\mathchoice{\left\{ #1 \ \middle| \ #2 \right\}}{\{ #1 \mid #2 \}}{\{ #1 \mid #2 \}}{\{ #1 \mid #2 \}}}}
\def\r#1{\text{\rm #1}}

\def\Bigv#1{\left| #1 \right|}
\def\v#1{{\mathchoice{\Bigv{#1}}{| #1 |}{| #1 |}{| #1 |}}}
\def\Bign#1{\left\| #1 \right\|}
\def\n#1{{\mathchoice{\Bign{#1}}{\| #1 \|}{\| #1 \|}{\| #1 \|}}}

\newcommand{\bN}{\mathbb{N}}

\newcommand{\bQ}{\mathbb{Q}}
\newcommand{\bR}{\mathbb{R}}

\newcommand{\bZ}{\mathbb{Z}}

\newcommand{\cC}{\mathscr{C}}

\newcommand{\cF}{\mathscr{F}}

\newcommand{\cP}{\mathscr{P}}

\newcommand{\cR}{\mathscr{R}}

\newcommand{\cV}{\mathscr{V}}
\newcommand{\cW}{\mathscr{W}}

\newcommand{\rA}{\r{A}}
\newcommand{\rB}{\r{B}}
\newcommand{\rC}{\r{C}}

\newcommand{\rM}{\r{M}}

\newcommand{\rP}{\r{P}}

\newcommand{\N}{\bN}
\newcommand{\Q}{\bQ}
\newcommand{\R}{\bR}
\newcommand{\Z}{\bZ}

\newcommand{\Cp}{\mathbb{C}_p}

\newcommand{\Ban}{\r{Ban}}

\newcommand{\Card}{\r{Card}}

\newcommand{\Hom}{\r{Hom}}

\newcommand{\id}{\r{id}}

\newcommand{\Isom}{\r{Isom}}

\newcommand{\Ord}{\r{Ord}}

\newcommand{\rank}{\r{rank}}

\newcommand{\Set}{\r{Set}}

\newcommand{\Suc}{\r{Suc}}

\newcommand{\rCbd}{\rC_{\r{bd}}}
\newcommand{\Rfl}{\r{Ref}}

\algnewcommand\algorithmicbreak{{\bf break}}
\algnewcommand\Break{\algorithmicbreak{}}
\algnewcommand\algorithmiccontinue{{\bf continue}}
\algnewcommand\Continue{\algorithmiccontinue{}}

\title{Structural Hierarchy of Reid Class of non-Archimedean Banach Spaces}
\author{Tomoki Mihara}
\date{}

\begin{document}

\maketitle
\begin{abstract}
Let $k$ be a complete valuation field. We formulate a class $\mathscr{R}$ of Banach $k$-vector spaces analogous to Reid class of Abelian groups. We formulate an analogue of the hierarchy of Reid class introduced by K.\ Eda, and verify a counterpart of the classification theorem of Reid class by K.\ Eda. As an application, we verify that the Banach $\mathbb{C}_p$-vector spaces
\begin{eqnarray*}
& & \ell^{\infty}(\mathbb{N},\mathbb{C}_p), \text{\rm C}_0(\mathbb{N},\mathbb{C}_p), \ell^{\infty}(\mathbb{N},\text{\rm C}_0(\mathbb{N},\mathbb{C}_p)), \text{\rm C}_0(\mathbb{N},\ell^{\infty}(\mathbb{N},\mathbb{C}_p)), \\
& & \ell^{\infty}(\mathbb{N},\text{\rm C}_0(\mathbb{N},\ell^{\infty}(\mathbb{N},\mathbb{C}_p))), \text{\rm C}_0(\mathbb{N},\ell^{\infty}(\mathbb{N},\text{\rm C}_0(\mathbb{N},\mathbb{C}_p))),
\end{eqnarray*}
and so on are all distinct, the Banach $\mathbb{C}_p$-vector space of bounded continuous functions $\mathbb{Q} \to \mathbb{C}_p$ and its dual Banach $\mathbb{C}_p$-vector spaces cannot be expressed by iterated application of bounded direct product and completed direct sum, and there is no left adjoint functor of the forgetful functor from $\mathscr{R}$ to the category of Banach $\mathbb{C}_p$-vector spaces.
\end{abstract}

\tableofcontents

\section{Introduction}
\label{Introduction}

Reid class is the smallest class of Abelian groups containing $\Z$ closed under direct product and direct sum (such that the cardinality of the index set is bounded by a fixed cardinal depending on the context). For example, $\Z^{\N}$, $\Z^{\oplus \N}$, $(\Z^{\oplus \N})^{\N}$, $(\Z^{\N})^{\oplus \N}$, $((\Z^{\N})^{\oplus \N})^{\N}$, $((\Z^{\oplus \N})^{\N})^{\oplus \N}$, and so on belong to Reid class. The study of Reid class arose with a natural question by G.\ Reid: ``Are these groups all distinct?''

\vs
B.\ Zimmermann-Huisgen affirmatively solved this problem in \cite{Zim79}. If we restrict the cardinality of index sets to non-$\omega_1$-measurable cardinals, {\L}o\'s--Eda theorem (cf.\ \cite{EL54}, \cite{Zee55}, and \cite{Eda82}) implies that every Abelian group in Reid class is reflexive with respect to the dual functor $\Hom(\cdot,\Z)$, and Reid class is closed under dual. In particular, non-reflexive Abelian groups such as $\Q$ do not belong to Reid class. How about the dual group of an Abelian group? How about the Abelian group of continuous functions $X \to \Z$ for a topological spaces $X$? Many problems arose, and have ever been successfully solved (cf.\ \cite{Iva80} and \cite{Eda83-2}).

\vs
A standard tool to analyse Reid class is Chase's lemma, which was originally stated as a theorem (cf.\ \cite{Cha62} Theorem 1.2). It roughly states that a homomorphism from a countable direct product to a countable direct sum essentially vanishes if we ignore finite components of the domain and the codomain and the divisible part of the codomain. Since the precise statement is a little complicated because of the use of a filter of right principal ideals of a non-commutative ring, we instead introduce a specialisation to $\Z$:

\begin{thm}[Chase's lemma]
\label{Chase's lemma}
Let $I$ be a countable set, $J$ a set, $(M_i)_{i \in I}$ a family of Abelian groups indexed by $I$, $(N_j)_{j \in J}$ a family of Abelian groups indexed by $J$, and $f$ a group homomorphism $\prod_{i \in I} M_i \to \bigoplus_{j \in J} N_j$. Then there exists a tuple $(m,I_0,J_0)$ of a positive integer $m$, a finite subset $I_0 \subset I$, and a finite subset $J_0 \subset J$ such that
\be
f \left( \ens{0}^{I_0} \times m \prod_{i \in I \setminus I_0} M_i \right) \subset \bigoplus_{j_0 \in J_0} N_{j_0} \oplus \bigcap_{n \in \N} n \bigoplus_{j \in J \setminus J_0} N_j.
\ee
\end{thm}

Theorem \ref{Chase's lemma} has various generalisations. For example, M.\ Dugas and B.\ Zimmermann-Huisgen extended Theorem \ref{Chase's lemma} to the non-$\omega_1$-measurable setting in \cite{DZH82} Theorem 2:

\begin{thm}[Dugas--Zimmermann-Huisgen's extension of Theorem \ref{Chase's lemma}]
\label{Dugas--Zimmermann-Huisgen's theorem}
Let $I$ be a set, $J$ a set, $(M_i)_{i \in I}$ a family of Abelian groups indexed by $I$, $(N_j)_{j \in J}$ a family of Abelian groups indexed by $J$, and $f$ a group homomorphism $\prod_{i \in I} M_i \to \bigoplus_{j \in J} N_j$. If $\# I$ is not $\omega_1$-measurable, then there exists a tuple $(m,I_0,J_0)$ of a positive integer $m$, a finite subset $I_0 \subset I$, and a finite subset $J_0 \subset J$ such that
\be
f \left( \ens{0}^{I_0} \times m \prod_{i \in I \setminus I_0} M_i \right) \subset \bigoplus_{j_0 \in J_0} N_{j_0} \oplus \bigcap_{n \in \N} n \bigoplus_{j \in J \setminus J_0} N_j.
\ee
\end{thm}

K.\ Eda further removed from Theorem \ref{Dugas--Zimmermann-Huisgen's theorem} the restriction of cardinality in \cite{Eda83-1} Theorem 2:

\begin{thm}[Eda's extension of Theorem \ref{Dugas--Zimmermann-Huisgen's theorem}]
\label{Eda's theorem}
Let $I$ be a set, $J$ a set, $(M_i)_{i \in I}$ a family of Abelian groups indexed by $I$, $(N_j)_{j \in J}$ a family of Abelian groups indexed by $J$, and $f$ a group homomorphism $\prod_{i \in I} M_i \to \bigoplus_{j \in J} N_j$. There exists a tuple $(m,U_0,J_0)$ of a positive integer $m$, a finite subset $U_0 \subset \beta_{\omega} I$, and a finite subset $J_0 \subset J$ such that
\be
f \left( m K_{H_0} \right) \subset \bigoplus_{j_0 \in J_0} N_{j_0} \oplus \bigcap_{n \in \N} n \bigoplus_{j \in J \setminus J_0} N_j,
\ee
where $\beta_{\omega} I$ denotes the set of $\omega_1$-complete ultrafilters of $I$ and $K_{H_0}$ denotes the subgroup $\set{(x_i)_{i \in I} \in \prod_{i \in I} M_i}{\forall \cF \in U_0[\set{i \in I}{x_i = 0} \in \cF]}$.
\end{thm}

K.\ Eda indicated in personal communication the expectation of the existence of a non-Archimedean counterpart of Theorem \ref{Eda's theorem} with application to analysis of a non-Archimedean counterpart of Reid class. As a preceding study, we formulated and verified a non-Archimedean analogue of Theorem \ref{Eda's theorem} in \cite{Mih26} Theorem 3.1. The aim of this paper is to formulate a non-Archimedean analogue $\cR$ of Reid class, and to apply a corollary of \cite{Mih26} Theorem 3.1 to analyse $\cR$. For example, we verify a non-Archimedean analogue of Eda's classification theorem of Abelian groups in Reid class by a hierarchy indexed by ordinals (cf.\ \cite{Eda83-2} Theorem 1). As an application, we show that $\ell^{\infty}(\N,k)$, $\rC_0(\N,k)$, $\ell^{\infty}(\N,\rC_0(\N,k))$, $\rC_0(\N,\ell^{\infty}(\N,k))$, $\ell^{\infty}(\N,\rC_0(\N,\ell^{\infty}(\N,k)))$, $\rC_0(\N,\ell^{\infty}(\N,\rC_0(\N,k)))$, and so on are all distinct if $\v{k}$ is dense in $[0,\infty)$ (cf.\ Corollary \ref{distinguishing}), the Banach $k$-vector space $\rCbd(\Q,k)$ of bounded continuous functions $\Q \to k$ and its dual Banach $k$-vector spaces cannot be expressed by iterated application of bounded direct product and completed direct sum if $\v{k}$ is dense in $[0,\infty)$ (cf.\ Theorem \ref{bounded continuous function} and Theorem \ref{measure space}), and there is no left adjoint functor of the forgetful functor from $\cR$ to the category of Banach $k$-vector spaces if $k$ is not spherically complete (cf.\ Theorem \ref{left adjoint}). We note that $k = \Cp$ satisfies the conditions.

\vs
We briefly explain contents of this paper. In \S \ref{Convention}, we introduce convention for this paper. In \S \ref{Reid Hierarchy}, we introduce $\cR$ and a non-Archimedean analogue of Eda's hierarchy. In \S \ref{Pure Reid Hierarchy}, we introduce a non-Archimedean analogue of proper $\Z$-kernel groups in the sense of \cite{Eda83-2}. In \S \ref{Reductive Summand}, we introduce a generalisation of the notion of direct summand, and verify a non-Archimedean analogue of Eda's classification theorem. In \S \ref{Counterexamples}, we give several examples of Banach $k$-vector spaces which do not belong to $\cR$, and study categorical properties of $\cR$.

\section{Convention}
\label{Convention}

We denote by $\omega$ the least transfinite ordinal $\aleph_0$, which is identical to the set of non-negative integers. We denote by $\Ord$ the class of ordinal numbers, by $\Suc \subset \Ord$ the subclass of successor ordinals, and by $\Card \subset \Ord$ the subclass of cardinal numbers. For an $\alpha \in \Suc$, we denote by $\alpha_{-}$ the predecessor of $\alpha$.

\vs
An ultrafilter $\cF$ is said to be {\it $\lambda$-complete} for a $\lambda \in \Card$ if $\bigcap_{U \in F} U \in \cF$ for any $F \in \cP_{< \lambda}(\cF) \setminus \ens{\emptyset}$. A set $X$ is said to be {\it $\lambda$-measurable} for a $\lambda \in \Card$ if $X$ is uncountable and admits a $\lambda$-complete non-principal ultrafilter. A $\mu \in \Card$ is said to be {\it measurable} if $\mu$ is $\mu$-measurable. We recall that a set $X$ is $\omega_1$-measurable if and only if there exists a measurable cardinal smaller than or equal to $\# X$, where $\omega_1$ denotes the least uncountable cardinal $\aleph_1$. Since the existence of a measurable cardinal is unprovable under $\textsf{ZFC}$ as long as $\textsf{ZFC}$ is consistent, so is the existence of an $\omega_1$-measurable set.

\vs
For a set $X$, we denote by $\# X$ its cardinality, by $\cP(X)$ the set of subsets of $X$, and by $\cP_{< \omega}(X) \subset \cP(X)$ the subset of finite subsets of $X$. Let $\kappa$ be a cardinal number or the class $\Card$. We denote by $\Set_{< \kappa}$ the class of sets $I$ with $\#I \in \kappa$.

\vs
For a class $X$ and a set $Y$, we denote by $X^Y$ the class of maps $Y \to X$. When we handle a sequence or a family $s$ indexed by a set $I$, we frequently use the map notation $s(i)$ instead of the subscript notation $s_i$ to point the entry at $i \in I$, in order to avoid massive use of subscripts. For a map $f$ and a subset $X'$ of its domain, we denote by $f \upharpoonright X'$ the restriction of $f$ to $X'$. For a set $X$ and a map $f \colon X \to \R_{\geq 0}$, we denote by $\sup_{x \in X} f(x)$ the supremum of the image of $f$ in $\R_{\geq 0} \sqcup \ens{\infty}$. In particular, $\sup_{x \in X} f(x)$ for the case $X = \emptyset$ is $0$ rather than $- \infty$ in this context.

\vs
For a set $X$, an $x \in X$, and a binary relation $R$ on $X$, we set $X_{R x} \coloneqq \set{x' \in X}{x' R x}$. We note that every $d \in \omega$ is identical to $\omega_{< d}$, and hence for a set $X$, $X^d$ formally means $X^{\omega_{< d}}$, which is naturally identified with the set of $d$-tuples in $X$.

\vs
Throughout this paper, $k$ denotes a complete valuation field. Although we do not assume the non-triviality of the valuation of $k$ first, we will concentrate on the case where $\v{k}$ is dense in $\R_{\geq 0}$ later. We denote by $\Ban(k)$ the class of Banach $k$-vector spaces.

\vs
For Banach spaces $V$ and $W$, we denote by $\Hom(V,W)$ the Banach $k$-vector space of bounded $k$-linear homomorphisms equipped with the operator norm, and set
\be
\Hom_{\leq 1}(V,W) \coloneqq \set{f \in \Hom(V,W)}{\n{f} \leq 1}.
\ee
When we refer to $\Ban(k)$ or its subclass, we always equip it with the category structure given by $\Hom_{\leq 1}$. In particular, an isomorphism of Banach $k$-vector spaces refers to an isomorphism in $\Ban(k)$ in this sense, i.e.\ an isometric $k$-linear isomorphism. For a subset $\cC \subset \Ban(k)$, we denote by $\Isom(\cC)$ the subclass $\set{V \in \Ban(k)}{\exists V' \in \cC[V \cong V']}$ of $\Ban(k)$, which is essentially small as a category by definition.

\vs
For a $V \in \Ban(k)$, we denote by $V^{\vee}$ the dual Banach $k$-vector space $\Hom(V,k)$. A $V \in \Ban(k)$ is said to be {\it reflexive} if the canonical morphism $V \to V^{\vee \vee}$ is an isomorphism.

\vs
Let $I$ be a set, and $\cW \in \Ban(k)^I$. For an $I' \in \cP(I)$, we denote by $\Pi_{I'} \cW$ the bounded direct product of $\cW \upharpoonright I'$, i.e.\ the Banach $k$-vector space whose underlying set is the set of maps $v \colon I' \to \bigsqcup_{i \in I'} \cW(i)$ such that $v(i) \in \cW(i)$ for any $i \in I'$ and $\sup_{i \in I'} \n{v(i)} < \infty$ and whose norm is the supremum norm, i.e.\ the map $\n{\cdot} \colon \Pi_{I'} \cW \to \R_{\geq 0}$ defined by
\be
\n{x} \coloneqq \sup_{i \in I'} \n{x(i)},
\ee
and by $\Sigma_{I'} \cW$ the completed direct sum of $\cW \upharpoonright I'$, i.e.\ the closed subspace of $\Pi_{I'} \cW$ given as
\be
& & \set{w \in \Pi_{I'} \cW}{\forall \epsilon \in \R_{> 0}[\exists I'_0 \in \cP_{< \omega}(I')[\forall i \in I' \setminus I'_0[\n{w(i)} < \epsilon]]]} \\
& = & \set{w \in \Pi_{I'} \cW}{\forall \epsilon \in \R_{> 0} \left[ \# \set{i \in I'}{\n{w(i)} \geq \epsilon} < \omega \right]}.
\ee
The reason why we use $\Sigma$ rather than $\bigoplus$ is because the use of $\Sigma$ makes the non-Archimedean analogue of Eda's hierarchy introduced later much easier to imagine, if the reader is familiar with arithmetic hierarchy and analytic hierarchy. We do not use $\Pi$ and $\Sigma$ for algebraic direct products or algebraic sums.

\vs
Throughout this paper, $\rA$ and $\rB$ denote distinct symbols which are either $\Pi$ or $\Sigma$. We use them in order to unify two descriptions such that one is given by the other one except for replacement of the occurrence of $\Pi$ and $\Sigma$ by $\Sigma$ and $\Pi$ respectively.

\vs
For an $I' \in \cP(I)$, we abuse the notation $\sigma_{I'}$ to indicate the zero extension $\rA_{I'} \cW \hookrightarrow \rA_I \cW$, and the notation $\pi^{I'}$ to indicate the canonical projection $\rA_{I} \cW \twoheadrightarrow \rA_{I'} \cW$.

\vs
For a set $J$, a $\cV \in \Ban^J$, an $I' \in \cP(I)$, a $J' \in \cP(J)$, and an $f \in \Hom(\rA_I \cW,\rB_J \cV)$, we abbreviate $\pi^{J'} \circ f \circ \sigma_{I'} \in \Hom(\rA_{I'} \cW,\rB_{J'} \cV)$ to $f_{I'}^{J'}$.

\vs
For a $V \in \Ban(k)$, we denote by $\ell^{\infty}(I,V)$ (resp.\ $\rC_0(I,V)$) the bounded direct product $\Pi_I (V)_{i \in I}$ (resp.\ the completed direct sum $\Sigma_I (V)_{i \in I}$) of the constant family $(V)_{i \in I}$.

\section{Reid Hierarchy}
\label{Reid Hierarchy}

We denote by $F_k^{\circ}$ the set $\set{\ell^{\infty}(\omega_{< n},k)}{n \in \omega}$, and set $F_k \coloneqq \Isom(F_k^{\circ})$.

\begin{dfn}
For a $\kappa \in \Card$ and an $\alpha \in \Ord$, we define subsets $\cR_{< \alpha,\kappa}^{\circ}$, $\Pi_{\alpha,\kappa}^{\circ}$, $\Sigma_{\alpha,\kappa}^{\circ}$, $\cR_{\alpha,\kappa}^{\circ}$, and $\Delta_{\alpha,\kappa}^{\circ}$ of $\Ban(k)$ in the following recursive way:
\be
\cR_{< \alpha,\kappa}^{\circ} & \coloneqq & \bigcup_{\beta \in \alpha} \cR_{\beta,\kappa}^{\circ} \\
\Pi_{\alpha,\kappa}^{\circ} & \coloneqq & F_k^{\circ} \cup \set{\Pi_I \cW}{I \in \kappa \land \cW \in (\cR_{< \alpha,\kappa}^{\circ})^I} \\
\Sigma_{\alpha,\kappa}^{\circ} & \coloneqq & F_k^{\circ} \cup \set{\Sigma_J \cV}{J \in \kappa \land \cV \in (\cR_{< \alpha,\kappa}^{\circ})^J} \\
\cR_{\alpha,\kappa}^{\circ} & \coloneqq & \Pi_{\alpha,\kappa}^{\circ} \cup \Sigma_{\alpha,\kappa}^{\circ} \\
\Delta_{\alpha,\kappa}^{\circ} & \coloneqq & \Pi_{\alpha,\kappa}^{\circ} \cap \Sigma_{\alpha,\kappa}^{\circ}
\ee
For a $\kappa \in \Card$ and an $\alpha \in \Ord$, we define essentially small subclasses $\cR_{< \alpha,\kappa}$, $\Pi_{\alpha,\kappa}$, $\Sigma_{\alpha,\kappa}$, $\cR_{\alpha,\kappa}$, and $\Delta_{\alpha,\kappa}$ of $\Ban(k)$ in the following way:
\be
\cR_{< \alpha,\kappa} & \coloneqq & \Isom(\cR_{< \alpha,\kappa}^{\circ}) \\
\Pi_{\alpha,\kappa} & \coloneqq & \Isom(\Pi_{\alpha,\kappa}^{\circ}) \\
\Sigma_{\alpha,\kappa} & \coloneqq & \Isom(\Sigma_{\alpha,\kappa}^{\circ}) \\
\cR_{\alpha,\kappa} & \coloneqq & \Isom(\cR_{\alpha,\kappa}^{\circ}) \\
\Delta_{\alpha,\kappa} & \coloneqq & \Isom(\Delta_{\alpha,\kappa}^{\circ})
\ee
For an $\alpha \in \Ord$, we define subclasses $\cR_{< \alpha,\Ord}$, $\Sigma_{< \alpha,\Ord}$, $\Pi_{\alpha,\Ord}$, $\Sigma_{\alpha,\Ord}$, $\cR_{\alpha,\Ord}$, and $\Delta_{\alpha,\Ord}$ of $\Ban(k)$ in the following way:
\be
\cR_{< \alpha,\Ord} & \coloneqq & \bigcup_{\kappa \in \Card} \cR_{< \alpha,\kappa} \\
\Pi_{\alpha,\Ord} & \coloneqq & \bigcup_{\kappa \in \Card} \Pi_{\alpha,\kappa} \\
\Sigma_{\alpha,\Ord} & \coloneqq & \bigcup_{\kappa \in \Card} \Sigma_{\alpha,\kappa} \\
\cR_{\alpha,\Ord} & \coloneqq & \bigcup_{\kappa \in \Card} \cR_{\alpha,\kappa} \\
\Delta_{\alpha,\Ord} & \coloneqq & \bigcup_{\kappa \in \Card} \Delta_{\alpha,\kappa}
\ee
\end{dfn}

In the following in this section, let $\kappa$ denote either a fixed infinite cardinal or the class $\Ord$. We omit the occurrence of ``$,\kappa$'' in subscripts in the notions defined above. Since $\Pi$ and $\Sigma$ preserve isomorphisms, we have
\be
\cR_{< \alpha} & = & \bigcup_{\beta \in \alpha} \cR_{\beta} \\
\Pi_{\alpha} & = & F_k \cup \set{V \in \Ban(k)}{\exists I \in \Set_{< \kappa}[\exists \cW \in \cR_{< \alpha}^I[V \cong \Pi_I \cW]]} \\
\Sigma_{\alpha} & = & F_k \cup \set{V \in \Ban(k)}{\exists J \in \Set_{< \kappa}[\exists \cV \in \cR_{< \alpha}^J[V \cong \Sigma_J \cV]]} \\
\cR_{\alpha} & = & \Pi_{\alpha} \cup \Sigma_{\alpha} \\
\Delta_{\alpha} & = & \Pi_{\alpha} \cap \Sigma_{\alpha}
\ee
for any $\alpha \in \Ord$ by transfinite induction on $\alpha$. We call these hierarchies {\it Reid hierarchy (with respect to $\kappa$)}, and study their basic properties.

\begin{prp}
\label{rank 0}
The equations
\be
\emptyset & = & \cR_{< 0} \\
F_k & = & \Pi_0 = \Sigma_0 = \cR_0.
\ee
hold.
\end{prp}

\begin{proof}
The first equation follows from the emptiness of the index set $0 = \omega_{< 0}$, and the second equation follows from the first equation.
\end{proof}

\begin{prp}
\label{ordered}
For any $\alpha \in \Ord$, the inclusions
\be
F_k \cup \cR_{< \alpha} \subset \Delta_{\alpha}
\ee
holds.
\end{prp}

\begin{proof}
The inclusion $F_k \subset \Delta_{\alpha}$ directly follows from the definitions of $\Sigma_{\alpha}$ and $\Pi_{\alpha}$. The inclusion $\cR_{< \alpha} \subset \Delta_{\alpha}$ follows from $V \cong \ell^{\infty}(\ens{\ast},V) = \rC_0(\ens{\ast},V)$ for any $V \in \cR_{< \alpha}$.
\end{proof}

By Proposition \ref{ordered}, we have
\be
\bigcup_{\alpha \in \Ord} \Pi_{\alpha} = \bigcup_{\alpha \in \Ord} \Sigma_{\alpha} = \bigcup_{\alpha \in \Ord} \cR_{\alpha}.
\ee
We denote this class by $\cR$, and call it {\it Reid class of Banach $k$-vector spaces (with respect to $\kappa$)}. For a $V \in \cR$, we denote by $\rank_{\cR}(V)$ the least ordinal $\alpha$ such that $V \in \cR_{\alpha}$, which exists by the definition of $\cR$ and the well-foundedness of $\Ord$. As the name indicates, $\cR$ is a non-Archimedean analogue of Reid class of Abelian groups, i.e.\ the smallest class $\cR$ of Abelian groups with $\Z \in \cR$ closed under isomorphism and also under direct product and direct sum. Indeed, the following characterisation holds:

\begin{prp}
Reid class of Banach $k$-vector spaces is the smallest class closed under isomorphism and also under bounded direct product and completed direct sum of families whose index sets belong to $\Set_{< \kappa}$, in the following sense:
\bi
\item[(1)] The class $\cR$ is closed under isomorphism and also under bounded direct product and completed direct sum of families whose index sets belong to $\Set_{< \kappa}$. 
\item[(2)] Let $\cC \subset \Ban(k)$ be a subclass with $k \in \cC$. If $\cC$ is closed under isomorphism and also under bounded direct product and completed direct sum of families whose index sets belong to $\Set_{< \kappa}$, then the inclusion $\cR \subset \cC$ holds.
\ei
\end{prp}

\begin{proof}
(1) Let $V \in \Ban(k)$ and suppose that there exists a $V' \in \cR$ with $V \cong V'$. By the definition of $\cR$, there exists an $\alpha \in \Ord$ such that $V' \in \cR_{\alpha}$. By $\cR_{\alpha} = \Isom(\cR_{\alpha}^{\circ})$, we obtain $V \in \cR_{\alpha} \subset \cR$.

\vs
Let $I \in \Set_{< \kappa}$ and $\cW \in \cR^I$. Set $\alpha \coloneqq (\sup \set{\rank_{\cR}(\cW(i))}{i \in I}) + 1$. We have $\cW \in \cR_{< \alpha}^I$. This implies $\Pi_I \cW \in \Pi_{\alpha} \subset \cR$ and $\Sigma_I \cW \in \Sigma_{\alpha} \subset \cR$.

\vs
(2) Since $\cC$ is closed under isomorphism and finite bounded direct product, we have $F_k \subset \cC$ by $k \in \cC$. It suffices to show $\cR_{\alpha} \setminus F_k \subset \cC$ for any $\alpha \in \Ord$ by transfinite induction on $\alpha$. Let $V \in \rA_{\alpha} \setminus F_k$. By the definition of $\rA_{\alpha}$, there exists a pair $(I,\cW)$ of an $I \in \Set_{< \kappa}$ and a $\cW \in \cR_{< \alpha}^I$ with $V \cong \rA_I \cW$. By induction hypothesis, we have $\cW \in \cC^I$. Since $\cC$ is closed under isomorphism and under bounded direct product and completed direct sum of families indexed by $I$, we have $V \in \cC$.
\end{proof}

Reid hierarchy can be stable and hence $\cR$ can be essentially small by the following reason:

\begin{prp}
\label{essentially small}
Suppose that $\kappa$ is a regular cardinal. Then the equality $\cR = \cR_{< \kappa}$ holds.
\end{prp}

\begin{proof}
It suffices to show $\cR_{\alpha} \subset \cR_{< \kappa}$ for any $\alpha \in \Ord$ by transfinite induction on $\alpha$. Let $V \in \rA_{\alpha}$. If $V \in F_k$, then we have $V \in \rA_0 \subset \cR_{< \kappa}$ by $0 \in \kappa$. Suppose $V \notin F_k$, i.e.\ there exists a pair $(I,\cW)$ of an $I \in \Set_{< \kappa}$ and a $\cW \in \cR_{< \alpha}^I$ with $V \cong \rA_I \cW$. By induction hypothesis, we have $\cW \in \cR_{< \kappa}^I$. In particular, we have $\rank_{\cR}(\cW(i)) \in \kappa$ for any $i \in I$. Set $\beta \coloneqq (\sup \set{\rank_{\cR}(\cW(i))}{i \in I}) + 1$. By $\# I \in \kappa$ and the regularity of $\kappa$, we have $\beta \in \kappa$. Therefore, we obtain $V \in \rA_{\beta} \subset \cR_{< \kappa}$ by $\cW \in \cR_{< \beta}^I$.
\end{proof}

\begin{prp}
\label{stability}
For any tuple $(\alpha,J,\cV)$ of an $\alpha \in \Ord$, a $J \in \Set_{< \kappa}$, and a $\cV \in \rA_{\alpha}^J$, the relation $\rA_J \cV \in \rA_{\alpha}$ holds.
\end{prp}

\begin{proof}
For each $j \in J$, we have $\cV(j) \in \rA_{\alpha}$, and take a pair $(I_j,\cW_j)$ of an $I_j \in \Set_{< \kappa}$ and a $\cW_j \in \cR_{< \alpha}^{I_j}$ with $\cV(j) \cong \rA_{I_j} \cW_j$. Set $I \coloneqq \bigsqcup_{j \in J} \ens{j} \times I_j$. We define a $\cW \in \cR_{< \alpha}^I$ by
\be
\cW((j,i)) \coloneqq \cW_j(i)
\ee
Then we obtain $\rA_J \cV \cong \rA_I \cW \in \rA_{\alpha}$.
\end{proof}

\begin{crl}
\label{finite product}
For any tuple $(\alpha,I,\cV)$ of an $\alpha \in \Ord$, an $I \in \Set_{< \omega}$, and a $\cV \in \rA_{\alpha}^I$, the relation $\Pi_I \cV = \Sigma_I \cV \in \rA_{\alpha}$ holds.
\end{crl}

\begin{proof}
The assertion follows from Proposition \ref{stability} and $\Pi_I \cV = \Sigma_I \cV = \rA_I \cV$.
\end{proof}

We give a non-Archimedean analogue of \cite{Eda83-2} Lemma 2:

\begin{prp}
\label{absorbing law}
For any $(V_0,V_1) \in F_k \times (\cR \setminus F_k)$, the relation $V_0 \times V_1 \cong V_1$ holds.
\end{prp}

\begin{proof}
It suffices to show that for any $(V_0,V_1) \in F_k \times (\cR_{\alpha} \setminus F_k)$, the relation $V_0 \times V_1 \cong V_1$ holds by transfinite induction on $\alpha$. Suppose $V_1 \in \rA_{\alpha} \setminus F_k$. By $V_1 \notin F_k$, there exists a pair $(I,\cW)$ of an $I \in \Set_{< \kappa}$ and a $\cW \in \cR_{< \alpha}^I$ with $V_1 \cong \rA_I \cW$. Replacing $I$ by $\set{i \in I}{\cW(i) \neq \ens{0}}$, we may assume $\cW(i) \neq \ens{0}$ for any $i \in I$. Set $I' \coloneqq \set{i \in I}{\cW(i) \in F_k}$.

\vs
First, suppose $I' = I$. Then we have $\rA_{I'} \cW \cong V_1 \notin F_k$, and hence $\# I' \geq \omega$. Set $J_0 \coloneqq \omega_{< \dim_k V_0}$ and $J_1 \coloneqq \bigsqcup_{i \in I'} \omega_{< \dim_k \cW(i)}$. We have $\# J_0 < \omega$ and $\# J_1 \geq \# I' \geq \omega$. This implies $\#(J_0 \sqcup J_1) = \# J_1$. We have
\be
V_0 \times V_1 \cong (\rA_{J_0} (k)_{j \in J_0}) \times (\rA_{J_1} (k)_{j \in J_1}) \cong \rA_{J_0 \sqcup J_1} (k)_{j \in J_0 \sqcup J_1} \cong \rA_{J_1} (k)_{j \in J_1} \cong V_1.
\ee
Next, suppose $I' \neq I$. Take an $i \in I \setminus I'$. By induction hypothesis, we have $V_0 \times \cW(i) \cong \cW(i)$. We obtain
\be
V_0 \times V_1 & \cong & V_0 \times \rA_I \cW \cong V_0 \times (\cW(i) \times \rA_{I \setminus \ens{i}} \cW) \\
& \cong & (V_0 \times \cW(i)) \times \rA_{I \setminus \ens{i}} \cW \cong \cW(i) \times \rA_{I \setminus \ens{i}} \cW \cong \rA_I \cW \cong V_1.
\ee
\end{proof}

\section{Pure Reid Hierarchy}
\label{Pure Reid Hierarchy}

Since $\cR$ is a non-Archimedean analogue of Reid class, its elements are analogous to $\Z$-kernel groups in the sense of \cite{Eda83-2}. We introduce a non-Archimedean analogue of the notion of proper $\Z$-kernel groups in the sense of \cite{Eda83-2}.

\begin{dfn}
For a $\kappa \in \Card$ and an $\alpha \in \Ord$, we define subsets $\rP \Pi_{< \alpha,\kappa}^{\circ}$, $\rP \Sigma_{< \alpha,\kappa}^{\circ}$, $\rP \cR_{< \alpha,\kappa}^{\circ}$, $\rP \Pi_{\alpha,\kappa}^{\circ}$, $\rP \Sigma_{\alpha,\kappa}^{\circ}$, and $\rP \cR_{\alpha,\kappa}^{\circ}$ of $\Ban(k)$ in the following recursive way:
\be
\rP \Pi_{< \alpha,\kappa}^{\circ} & \coloneqq & \bigcup_{\beta \in \alpha} \rP \Pi_{\beta,\kappa}^{\circ} \\
\rP \Sigma_{< \alpha,\kappa}^{\circ} & \coloneqq & \bigcup_{\beta \in \alpha} \rP \Sigma_{\beta,\kappa}^{\circ} \\
\rP \cR_{< \alpha,\kappa}^{\circ} & \coloneqq & \bigcup_{\beta \in \alpha} \rP \cR_{\beta,\kappa}^{\circ} \\
\rP \Pi_{\alpha,\kappa}^{\circ} & \coloneqq & 
\left\{
\begin{array}{ll}
\ens{k} & (\alpha = 0) \\
\displaystyle\set{\Pi_I \cW}{
\begin{array}{ll}
I \in \kappa \land \cW \in (\rP \cR_{< \alpha,\kappa}^{\circ})^I \land \\
\forall \beta \in \alpha \left[ \# \set{i \in I}{\cW(i) \in \rP \Sigma_{< \alpha,\kappa}^{\circ} \setminus \rP \Sigma_{< \beta,\kappa}^{\circ}} \geq \omega \right]
\end{array}
} & (\alpha \neq 0)
\end{array}
\right. \\
\rP \Sigma_{\alpha,\kappa}^{\circ} & \coloneqq & 
\left\{
\begin{array}{ll}
\ens{k} & (\alpha = 0) \\
\displaystyle\set{\Sigma_J \cV}{
\begin{array}{ll}
J \in \kappa \land \cV \in (\rP \cR_{< \alpha,\kappa}^{\circ})^J \land \\
\forall \beta \in \alpha \left[\# \set{j \in J}{\cV(j) \in \rP \Pi_{< \alpha,\kappa}^{\circ} \setminus \rP \Pi_{< \beta,\kappa}^{\circ}} \geq \omega \right]
\end{array}
} & (\alpha \neq 0)
\end{array}
\right. \\
\rP \cR_{\alpha,\kappa}^{\circ} & \coloneqq & \rP \Pi_{\alpha,\kappa}^{\circ} \cup \rP \Sigma_{\alpha,\kappa}^{\circ}
\ee
For a $\kappa \in \Card$ and an $\alpha \in \Ord$, we define essentially small subclasses $\rP \Pi_{< \alpha,\kappa}$, $\rP \Sigma_{< \alpha,\kappa}$, $\rP \cR_{< \alpha,\kappa}$, $\rP \Pi_{\alpha,\kappa}$, $\rP \Sigma_{\alpha,\kappa}$, and $\rP \cR_{\alpha,\kappa}$ of $\Ban(k)$ in the following way:
\be
\rP \Pi_{< \alpha,\kappa} & \coloneqq & \Isom(\rP \Pi_{< \alpha,\kappa}^{\circ}) \\
\rP \Sigma_{< \alpha,\kappa} & \coloneqq & \Isom(\rP \Sigma_{< \alpha,\kappa}^{\circ}) \\
\rP \cR_{< \alpha,\kappa} & \coloneqq & \Isom(\rP \cR_{< \alpha,\kappa}^{\circ}) \\
\rP \Pi_{\alpha,\kappa} & \coloneqq & \Isom(\rP \Pi_{\alpha,\kappa}^{\circ}) \\
\rP \Sigma_{\alpha,\kappa} & \coloneqq & \Isom(\rP \Sigma_{\alpha,\kappa}^{\circ}) \\
\rP \cR_{\alpha,\kappa} & \coloneqq & \Isom(\rP \cR_{\alpha,\kappa}^{\circ})
\ee
For an $\alpha \in \Ord$, we define subclasses $\rP \Pi_{< \alpha,\Ord}$, $\rP \Sigma_{< \alpha,\Ord}$, $\rP \cR_{< \alpha,\Ord}$, $\rP \Pi_{\alpha,\Ord}$, $\rP \Sigma_{\alpha,\Ord}$, and $\rP \cR_{\alpha,\Ord}$ of $\Ban(k)$ in the following way:
\be
\rP \Pi_{< \alpha,\Ord} & \coloneqq & \bigcup_{\kappa \in \Card} \rP \Pi_{< \alpha,\kappa} \\
\rP \Sigma_{< \alpha,\Ord} & \coloneqq & \bigcup_{\kappa \in \Card} \Sigma \cR_{< \alpha,\kappa} \\
\rP \cR_{< \alpha,\Ord} & \coloneqq & \bigcup_{\kappa \in \Card} \rP \cR_{< \alpha,\kappa} \\
\rP \Pi_{\alpha,\Ord} & \coloneqq & \bigcup_{\kappa \in \Card} \rP \Pi_{\alpha,\kappa} \\
\rP \Sigma_{\alpha,\Ord} & \coloneqq & \bigcup_{\kappa \in \Card} \rP \Sigma_{\alpha,\kappa} \\
\rP \cR_{\alpha,\Ord} & \coloneqq & \bigcup_{\kappa \in \Card} \rP \cR_{\alpha,\kappa}
\ee
\end{dfn}

In the following in this paper, let $\kappa$ denote either a fixed infinite cardinal or the class $\Ord$. We omit the occurrence of ``$,\kappa$'' in subscripts in the notions defined above.

\begin{dfn}
For any $(\alpha,I) \in \Ord \times \Set_{< \kappa}$, we denote by $\rP \cR_{< \alpha}^{I \to \rA}$ the class of $\cW \in \rP \cR_{< \alpha}^I$ such that for any $\beta \in \alpha$, the inequality $\# \set{i \in I}{\cW(i) \in \rP \rA_{< \alpha} \setminus \rP \rA_{< \beta}} \geq \omega$ holds.
\end{dfn}

Since $\Pi$ and $\Sigma$ preserve isomorphisms, we have
\be
\rP \Pi_{< \alpha} & = & \bigcup_{\beta \in \alpha} \rP \Pi_{\beta} \subset \cR_{< \alpha} \\
\rP \Sigma_{< \alpha} & = & \bigcup_{\beta \in \alpha} \rP \Sigma_{\beta} \subset \cR_{< \alpha} \\
\rP \cR_{< \alpha} & = & \bigcup_{\beta \in \alpha} \rP \cR_{\beta} \subset \cR_{< \alpha} \\
\rP \Pi_{\alpha} & = & 
\left\{
\begin{array}{ll}
\Isom(\ens{k}) & (\alpha = 0) \\
\displaystyle\set{V \in \Ban(k)}{\exists I \in \Set_{< \kappa}[\exists \cW \in \rP \cR_{< \alpha}^{I \to \Sigma} [V \cong \Pi_I \cW]]} & (\alpha \neq 0)
\end{array}
\right. \subset \Pi_{\alpha} \\
\rP \Sigma_{\alpha} & = &
\left\{
\begin{array}{ll}
\Isom(\ens{k}) & (\alpha = 0) \\
\displaystyle\set{V \in \Ban(k)}{\exists J \in \Set_{< \kappa}[\exists \cV \in \rP \cR_{< \alpha}^{J \to \Pi}[V \cong \Sigma_J \cV]]} & (\alpha \neq 0)
\end{array}
\right. \subset \Sigma_{\alpha} \\
\rP \cR_{\alpha} & = & \rP \Pi_{\alpha} \cup \rP \Sigma_{\alpha} \subset \cR_{\alpha}
\ee
for any $\alpha \in \Ord$ by transfinite induction on $\alpha$. We call these hierarchies {\it pure Reid hierarchy (with respect to $\kappa$)}. Although the definition might not look like a direct analogue of that of proper $\Z$-kernel groups because of the lack of the case classification on whether $\alpha$ is a limit ordinal or not, the minor modification is just for simplicity (cf.\ Proposition \ref{degeneracy of mixed} and Proposition \ref{pure convergent Limit stabilty}). We set
\be
\rP \Pi & \coloneqq & \bigcup_{\alpha \in \Ord} \rP \Pi_{\alpha} \\
\rP \Sigma & \coloneqq & \bigcup_{\alpha \in \Ord} \rP \Sigma_{\alpha} \\
\rP \cR & \coloneqq & \bigcup_{\alpha \in \Ord} \rP \cR_{\alpha},
\ee
and call $\rP \cR$ {\it pure Reid class of Banach $k$-vector spaces (with respect to $\kappa$)}. For a $V \in \rP \Pi$ (resp.\ $\rP \Sigma$), we denote by $\rank_{\rP \Pi}(V)$ (resp.\ $\rank_{\rP \Sigma}(V)$, $\rank_{\rP \cR}(V)$) the least ordinal $\alpha$ such that $V \in \rP \Pi_{\alpha}$ (resp.\ $\rP \Sigma_{\alpha}$, $\rP \cR_{\alpha}$), which exists by the definition of $\rP \Pi$ (resp.\ $\rP \Sigma$, $\rP \cR$) and the well-foundedness of $\Ord$. We study basic properties of pure Reid hierarchy.

\begin{prp}
\label{pure rank 0}
Let $\alpha \in \Ord$. Then the following hold:
\bi
\item[(1)] If $\alpha = 0$, then the relation $\rP \Pi_{\alpha} = \rP \Sigma_{\alpha} = \rP \cR_{\alpha} = \Isom(\ens{k}) \subset F_k$ holds.
\item[(2)] If $\alpha \neq 0$, then the equality $\rP \cR_{\alpha} \cap F_k = \emptyset$ holds.
\ei
\end{prp}

\begin{proof}
The assertion (1) follows directly from the definition. We show the assertion (2) by transfinite induction on $\alpha$. Let $V \in \rP \rA_{\alpha}$. Take a pair $(I,\cW)$ of an $I \in \Set_{< \kappa}$ and a $\cW \in \rP \cR_{< \alpha}^{I \to \rB}$ with $V \cong \rA_I \cW$. By (1) and induction hypothesis, we have $\cW(i) \neq \ens{0}$ for any $i \in I$. By $\cW \in \rP \cR_{< \alpha}^{I \to \rB}$, we have $\# I \geq \omega$. Therefore, we obtain $\dim_k V \geq \omega$, i.e.\ $V \notin F_k$.
\end{proof}

\begin{prp}
\label{pure convergent successor stabilty}
For any tuple $(\alpha,J,\cV)$ of an $\alpha \in \Suc$, a $J \in \Set_{< \kappa}$, and a $\cV \in \rP \cR_{< \alpha}^{J \to \rA}$, the relation $\rA_J \cV \in \rP \rA_{\alpha_{-}}$ holds.
\end{prp}

\begin{proof}
We follow the convention in the proof of Proposition \ref{stability} except for the additional condition $\cW_j \in \rP \cR_{< \rank_{\rP \rA}(\cV(j))}^{I_j \to \rB}$. It suffices to show $\cW \in \rP \cR_{< \alpha_{-}}^{I \to \rB}$. Let $\beta \in \alpha_{-}$. Set $J_0 \coloneqq \set{j \in J}{\cV(j) \in \rP \rA_{< \alpha} \setminus \rP \rA_{< \alpha_{-}}}$ and $I_0 \coloneqq \set{(j,i) \in I}{\cW((j,i)) \in \rP \rB_{< \alpha_{-}} \setminus \rP \rB_{< \beta}}$. It suffices to show $\# I_0 \geq \omega$.

\vs
For each $j \in J_0$, we have $\rank_{\rP \rA}(\cV(j)) = \alpha_{-}$, and hence there exists an $i_j \in I_j$ such that $\cW_j(i_j) \in \rP \rB_{\alpha_{-}} \setminus \rP \rB_{\beta}$ by $\cW_j \in \rP \cR_{< \rank_{\rP \rA}(\cV(j))}^{I_j \to \rB}$ and $\beta \in \alpha_{-}$. We have $\set{(j,i_j)}{j \in J_0} \subset I_0$, and hence $\# I_0 \geq \# J_0 \geq \omega$ by $\cV \in \rP \cR_{< \alpha}^{J \to \rA}$.
\end{proof}

\begin{prp}
\label{pure finite product}
For any tuple $(\alpha_0,\alpha_1,V_0,V_1)$ of an $\alpha_0 \in \Ord$, an $\alpha_1 \in \Ord$, a $V_0 \in \rP \rA_{\alpha_0}$, and a $V_1 \in \rP \rA_{\alpha_1}$, the relation $V_0 \times V_1 \in \rP \rA_{\max \ens{\alpha_0,\alpha_1}}$ holds.
\end{prp}

\begin{proof}
For each $h \in \ens{0,1}$, take a pair $(I_h,\cW_h)$ of an $I_h \in \Set_{< \kappa}$ and a $\cW_h \in \rP \cR_{< \alpha_h}^{I \to \rB}$ with $V_h \cong \rA_{I_h} \cW_h$. Set $\alpha \coloneqq \max \ens{\alpha_0,\alpha_1}$, $V \coloneqq V_0 \times V_1$, and $I \coloneqq I_0 \sqcup I_1$. We define $\cW \in \rP \cR_{< \alpha}^I$ by
\be
\cW(i) \coloneqq
\left\{
\begin{array}{ll}
\cW_0(i) & (i \in I_0) \\
\cW_1(i) & (i \in I_1)
\end{array}
\right..
\ee
We have $V \cong (\rA_{I_0} \cW_0) \times (\rA_{I_1} \cW_1) \cong \rA_I \cW$. It suffices to show $\cW \in \rP \cR_{< \alpha}^{I \to \rB}$. Let $\beta \in \alpha$. By the definition of $\alpha$, we have $\alpha = \alpha_h$ for some $h \in \ens{0,1}$. We obtain
\be
\# \set{i \in I}{\cW(i) \in \rB_{< \alpha} \setminus \rB_{< \beta}} \geq \# \set{i \in I_h}{\cW_h(i) \in \rB_{< \alpha_h} \setminus \rB_{< \beta}} \geq \omega
\ee
by $\beta \in \alpha = \alpha_h$, and hence $\cW \in \rP \cR_{< \alpha}^{I \to \rB}$.
\end{proof}

\begin{prp}
\label{mixed finite product}
For any tuple $(\alpha,\beta,V,W)$ of an $\alpha \in \Ord$, a $\beta \in \alpha$, a $V \in \rP \rA_{\alpha}$, and a $\rP \rB_{\beta}$, the relation $V \times W \in \rP \rA_{\alpha}$ holds.
\end{prp}

\begin{proof}
Take a pair $(I_0,\cW_0)$ of an $I_0 \in \Set_{< \kappa}$ and a $\cW_0 \in \rP \cR_{< \alpha}^{I \to \rB}$ with $V \cong \rA_{I_0} \cW_0$. Set $I \coloneqq I_0 \sqcup \ens{\ast}$. We define $\cW \in \rP \cR_{< \alpha}^I$ by
\be
\cW(i) \coloneqq
\left\{
\begin{array}{ll}
\cW_0(i) & (i \in I_0) \\
W & (i = \ast)
\end{array}
\right..
\ee
We have $V \times W \cong (\rA_{I_0} \cW) \times \cW(\ast) \cong \rA_I \cW$. For any $\beta \in \alpha$, we have
\be
\# \set{i \in I}{\cW(i) \in \rB_{< \alpha} \setminus \rB_{< \beta}} \geq \# \set{i \in I_0}{\cW_0(i) \in \rB_{< \alpha} \setminus \rB_{< \beta}} \geq \omega.
\ee
This implies $\cW \in \rP \cR_{< \alpha}^{I \to \rB}$ and hence $V \times W \in \rP \rA_{\alpha}$.
\end{proof}

We introduce a non-Archimedean analogue of the notion of improper $\Z$-kernel groups. For each $\alpha \in \Ord$, we set
\be
\rM \cR_{\alpha} \coloneqq \set{V \in \Ban(k)}{\exists (V_{\Pi},V_{\Sigma}) \in \rP \Pi_{\alpha} \times \rP \Sigma_{\alpha}[V \cong V_{\Pi} \times V_{\Sigma}]}.
\ee
For each $\alpha \in \Ord$, we set
\be
\rM \cR_{< \alpha} \coloneqq \bigcup_{\beta \in \alpha} \rM \cR_{\alpha}.
\ee
This gives the ``mixed part'' of Reid hierarchy. The following two propositions indicate that the mixed part matters only when we consider a limit ordinal $\alpha$:

\begin{prp}
\label{degeneracy of mixed}
Let $\alpha \in \Ord$. The the following hold:
\bi
\item[(1)] The inclusion $\rM \cR_{\alpha} \subset \Delta_{\alpha + 1} \subset \cR_{\alpha + 1}$ holds.
\item[(2)] If $\alpha = 0$, then the inclusion $\rM \cR_{\alpha} \subset F_k$ holds.
\item[(3)] If $\alpha$ is a non-zero limit ordinal, then the inclusion $\rM \cR_{< \alpha} \subset \cR_{< \alpha}$ holds.
\ei
\end{prp}

\begin{proof}
The assertion (1) follows from Proposition \ref{ordered} and Corollary \ref{finite product}. The assertion (2) follows from Proposition \ref{pure rank 0}. The assertion (3) follows from the assertion (1).
\end{proof}

\begin{prp}
\label{pure convergent Limit stabilty}
For any tuple $(\alpha,I,\cW)$ of an $\alpha \in \Ord \setminus \ens{0}$, an $I \in \Set_{< \kappa}$, and a $\cW \in \rP \cR_{< \alpha}^I$ with $\sup \set{\rank_{\rP \cR}(\cW(i))}{i \in I} = \alpha$, the relation $\rA_I \cW \in \rP \rA_{\alpha}$ holds.
\end{prp}

\begin{proof}
By $\cW \in \rP \cR_{< \alpha}^I$ and $\sup \set{\rank_{\rP \cR}(\cW(i))}{i \in I} = \alpha$, $\alpha$ is a limit ordinal. Set $V \coloneqq \rA_I \cW$. If $\cW \in \rP \cR_{< \alpha}^{I \to \rB}$, then we have $V \in \rP \rA_{\alpha}$ by definition. Suppose $\cW \notin \rP \cR_{< \alpha}^{I \to \rB}$, i.e.\ there exists an $\alpha' \in \alpha$ such that $\# \set{i \in I}{\cW(i) \in \rP \rB_{< \alpha} \setminus \rP \rB_{< \alpha'}} < \omega$.

\vs
Set $I_{\rB} \coloneqq \set{i \in I}{\cW(i) \in \rP \rB_{< \alpha}}$ and $I_{\rA} \coloneqq I \setminus I_{\rB}$. For each $i \in I_{\rA}$, we have $\cW(i) \in \rP \cR_{< \alpha} \setminus \rP \rB_{< \alpha} \subset \rP \rA_{< \alpha}$, and take a pair $(I_i,\cW_i)$ of an $I_i \in \Set_{< \kappa}$ and a $\cW_i \in \rP \cR_{< \rank_{\rP \rA}(\cW(i))}^{I_i \to \rB}$ with $\cW(i) \cong \rA_{I_i} \cW_i$. Set $I' \coloneqq (\bigsqcup_{i \in I_{\rA}} \ens{i} \times I_i) \sqcup (\bigsqcup_{i \in I_{\rB}} \ens{(i,i)})$ and define $\cW' \in \rP \cR_{< \alpha}^{I'}$ by
\be
\cW'((i,i')) \coloneqq
\left\{
\begin{array}{ll}
\cW_i(i') & (i \in I_{\rA}) \\
\cW(i) & (i \in I_{\rB})
\end{array}
\right..
\ee
We have $V = \rA_I \cW \cong \rA_{I'} \cW'$, and hence it suffices to show $\cW' \in \rP \cR_{< \alpha}^{I \to \rB}$. Let $\beta \in \alpha$. Set $I'_0 \coloneqq \set{(i,i') \in I'}{\cW'((i,i')) \in \rP \rB_{< \alpha} \setminus \rP \rB_{< \beta}}$. It suffices to show $\# I'_0 \geq \omega$

\vs
Set $\beta' \coloneqq \max \ens{\alpha',\beta} + 1$. Since $\alpha$ is a non-zero limit ordinal, we have $\beta' \in \alpha$. Set $I_0 \coloneqq \set{i \in I}{\cW(i) \notin \rP \cR_{< \beta'}}$. By $\sup \set{\rank_{\rP \cR}(\cW(i))}{i \in I} = \alpha$, we have $\# I_0 \geq \omega$. By the choice of $\alpha'$, we have $\#(I_0 \cap I_{\rB}) < \omega$, and hence $\#(I_0 \cap I_{\rA}) \geq \omega$.

\vs
Let $i \in I_0 \cap I_{\rA}$. By $i \notin I_{\rB}$ and $\cW(i) \in \rP \cR_{< \alpha}$, we have $\rank_{\rP \cR}(\cW(i)) = \rank_{\rP \rA}(\cW(i)) < \alpha$. By $i \in I_0$, we have $\rank_{\rP \cR}(\cW(i)) \geq \beta' > \beta$, and hence
\be
\# \set{i' \in I_i}{\cW_i(i') \in \rP \rB_{< \rank_{\rP \rA}(\cW(i))} \setminus \rP \rB_{< \beta}} \geq \omega > 0.
\ee
Take an $h_i \in I_i$ with $\cW_i(h_i) \in \rP \rB_{< \rank_{\rP \rA}(\cW(i))} \setminus \rP \rB_{< \beta}$. Then we have $(i,h_i) \in I'_0$. We conclude $\# I'_0 \geq \#(I \cap I_{\rA}) \geq \omega$.
\end{proof}

\begin{prp}
\label{pure non-convergent stabilty}
For any tuple $(\alpha,J,\cV)$ of an $\alpha \in \Suc$, a $J \in \Set_{< \kappa} \setminus \ens{\emptyset}$, and a $\cV \in \rP \rA_{< \alpha}^J$, the relation $\rA_J \cV \in \rP \rA_{< \alpha}$ holds.
\end{prp}

\begin{proof}
For a $\beta \in \Ord$, set $J_{\geq \beta} \coloneqq \set{j \in J}{\cV(j) \notin \rP \rA_{< \beta}}$. We have $\# J_{\geq \alpha} = 0 < \omega$. Let $\alpha_0$ denote the minimum of a $\beta \in \Ord$ with $\# J_{\geq \beta} <\omega$. Set $J_0 \coloneqq J_{\geq \alpha_0}$ and $J_1 \coloneqq J \setminus J_0$. By the definition of $\alpha_0$, we have $\# J_0 < \omega$. Set $V_0 \coloneqq \rA_{J_0} \cV$ and $V_1 \coloneqq \rA_{J_1} \cV$.

\vs
If $J_0 = \emptyset$, then we have $V_0 = \ens{0}$. If $J_0 \neq \emptyset$, we have $V_0 \in \rP \rA_{< \alpha}$ by Proposition \ref{pure finite product}.

\vs
If $\alpha_0 = 0$, then we have $J_1 = \emptyset$, and hence $V_1 = \ens{0}$. If $\alpha_0 \neq 0$, then we have $\cW \upharpoonright J_1 \in \rP \cR_{< \alpha_0}^{J_1 \to \rA}$ by the minimality of $\alpha_0$, and hence $V_1 \in \rP \rA_{< \alpha}$ by Proposition \ref{pure convergent successor stabilty} for the case $\alpha_0 \in \Suc$ and Proposition \ref{pure convergent Limit stabilty} for the case $\alpha_0 \notin \Suc \cup \ens{0}$.

\vs
By $J \neq \emptyset$, we have either $J_0 \neq \emptyset$ or $J_1 \neq \emptyset$. Therefore, $V \cong V_0 \times V_1$ is a product of two Banach $k$-vector spaces both of which are zero or belong to $\rP \rA_{< \alpha}$ but one of which is not zero, and hence belongs to $\rP \rA_{< \alpha}$ by Proposition \ref{pure finite product}.
\end{proof}

\begin{prp}
\label{successor non-pure}
For any tuple $(\alpha,V,I,\cW)$ of an $\alpha \in \Ord \setminus \ens{0}$, a $V \in \cR_{\alpha} \setminus (\cR_{< \alpha} \cup \rP \rA_{\alpha})$, an $I \in \Set_{< \kappa}$, and a $\cW \in \rP \cR_{< \alpha}^I$ with $V \cong \rA_I \cW$, the relations $\alpha \in \Suc$ and $V \in \rM \cR_{\alpha_{-}}$ hold.
\end{prp}

\begin{proof}
Set $I_{\rA} \coloneqq \set{i \in I}{\cW(i) \in \rP \rA}$, $I_{\rB} \coloneqq I \setminus I_{\rA}$, $\beta_1 \coloneqq \sup \set{\rank_{\rP \rB}(\cW(i))}{i \in I_{\rB}}$, $V_0 \coloneqq \rA_{I_{\rA}} \cW$, and $V_1 \coloneqq \rA_{I_{\rB}} \cW$. By Proposition \ref{pure non-convergent stabilty}, we have $V_0 \in \rP \rA_{< \alpha}$. If $V_1 \in \rP \rA_{< \alpha + 1}$, then we have $V \cong V_0 \times V_1 \in \rP \rA_{< \alpha + 1}$ by Proposition \ref{pure finite product}, which contradicts $V \notin \cR_{< \alpha} \cup \rP \rA_{\alpha}$. Therefore, we obtain $V_1 \notin \rP \rA_{< \alpha + 1}$.

\vs
Set $I_{\geq \beta} \coloneqq \set{i \in I_{\rB}}{\cW(i) \notin \rP \rB_{< \beta}}$ for a $\beta \in \alpha$, and $I_{\rB,1} \coloneqq I_{\geq \beta_1}$. If $\# I_{\rB,1} = 0$, then we have $\cW \upharpoonright I_{\rB} \in \rP \cR_{< \beta_1}^{I \to \rB}$, and hence $V_1 \in \rP \rA_{\beta_1} \subset \rP \rA_{< \alpha + 1}$, which leads to contradiction. We obtain $\# I_{\rB,1} > 0$, and hence $\beta_1 < \alpha$. If $\# I_{\rB,1} \geq \omega$, then we have $\cW \upharpoonright I_{\rB} \in \rP \cR_{< \beta_1 + 1}^{I \to \rB}$, and hence $V_1 \in \rP \rA_{\beta_1 + 1} \subset \rP \rA_{< \alpha + 1}$, which leads to contradiction.

\vs
We obtain $0 < \# I_{\rB,1} < \omega$. Let $\beta_0$ denote the minimum of a $\beta \in \alpha$ such that $\# I_{\geq \beta} < \omega$. Set $I_{\rB,0} \coloneqq I_{\geq \beta_0}$, $V'_1 \coloneqq \rA_{I_{\rB} \setminus I_{\rB,0}} \cW$, $V_{\rA} \coloneqq V_0 \times V'_1$ and $V_{\rB} \coloneqq \rA_{I_{\rB,0}} \cW$. By $\emptyset \neq I_{\rB,1} \subset I_{\rB,0}$ and Proposition \ref{pure finite product}, we have $V_{\rB} \in \rP \rB_{< \alpha}$.

\vs
We show $V_{\rA} \in \rP \rA_{< \alpha}$. If $\beta_0 = 0$, then we have $I_{\rB,0} = I_{\rB}$ and hence $V_{\rA} \cong V_0 \in \rP \rA_{< \alpha}$. Suppose $\beta_0 > 0$. By the minimality of $\beta_0$, we have $\cW \upharpoonright (I_{\rB} \setminus I_{\rB,0}) \in \rP \cR_{< \beta_0}^{I_{\rB} \setminus I_{\rB,0} \to \rB}$. Therefore, we obtain $V'_1 \in \rP \rA_{\beta_0}$, and hence $V_{\rA} \in \rP \rA_{< \alpha}$ by Proposition \ref{pure finite product}.

\vs
Set $r_{\rA} \coloneqq \rank_{\rP \rA}(V_{\rA})$ and $r_{\rB} \coloneqq \rank_{\rP \rB}(V_{\rB})$. If $r_{\rA} \neq r_{\rB}$, then we have $V \cong V_{\rA} \times V_{\rB} \in \rP \cR_{\max \ens{r_0,r_1}}$ by Proposition \ref{mixed finite product}, which contradicts $V \notin \cR_{< \alpha}$. Therefore, we obtain $r_{\rA} = r_{\rB}$, and hence $V \in \rM \cR_{r_{\rA}}$. This implies $\alpha = r_{\rA} + 1$ by $V \notin \cR_{< \alpha}$ and Proposition \ref{degeneracy of mixed}.
\end{proof}

We show a non-Archimedean analogue of \cite{Eda83-2} Lemma 3:

\begin{thm}
\label{mixed decomposition}
For any $\alpha \in \Ord \setminus \ens{0}$, the equalities
\be
\Pi_{\alpha} & = & F_k \cup \cR_{< \alpha} \cup \rM \cR_{< \alpha} \cup \rP \Pi_{\alpha} \\
\Sigma_{\alpha} & = & F_k \cup \cR_{< \alpha} \cup \rM \cR_{< \alpha} \cup \rP \Sigma_{\alpha} \\
\cR_{\alpha} & = & F_k \cup \cR_{< \alpha} \cup \rM \cR_{< \alpha} \cup \rP \cR_{\alpha}
\ee
hold.
\end{thm}

\begin{proof}
Since the third equality follows from the first two equalities, it suffices to show the first two equalities by transfinite induction on $\alpha$. By Proposition \ref{ordered}, Corollary \ref{finite product}, and Proposition \ref{degeneracy of mixed} (1), it suffices to show that $\rA_{\alpha} \setminus (F_k \cup \cR_{< \alpha} \cup \rP \rA_{\alpha}) \subset \rM \cR_{< \alpha}$.

\vs
Let $V \in \rA_{\alpha} \setminus (F_k \cup \cR_{< \alpha} \cup \rP \rA_{\alpha})$. By $V \in \rA_{\alpha} \setminus F_k$, there exists a pair $(I,\cW)$ of an $I \in \Set_{< \kappa}$ and a $\cW \in \cR_{< \alpha}^I$ with $V \cong \rA_I \cW$. Replacing $I$ by $\set{i \in I}{\cW(i) \neq \ens{0}}$, we may assume $\cW(i) \neq \ens{0}$ for any $i \in I$. Set $I_0 \coloneqq \set{i \in I}{\cW(i) \in F_k}$ and $I_1 \coloneqq I \setminus I_0$.

\vs
If $\# I_0 \geq \omega$, we replace $I$ by $\ens{\ast} \sqcup I_1$ and $\cW$ by the map $\ens{\ast} \sqcup I_1 \to \cR_{< \alpha} \setminus F_k$ defined by
\be
i \mapsto 
\left\{
\begin{array}{ll}
\rA_{I_0} \cW & (i = \ast) \\
\cW(i) & (i \in I_1)
\end{array}
\right.
\ee
so that the resulting $I_0$ becomes $\emptyset$. Therefore, we may assume $\# I_0 < \omega$ and hence $\rA_{I_0} \cW \in F_k$. We have $I_1 \neq \emptyset$ by $V \notin F_k$, and hence $\rA_{I_1} \cW \notin F_k$ by $\cW(i) \notin F_k$ for any $i \in I_1$. We obtain
\be
V \cong \rA_I \cW \cong (\rA_{I_0} \cW) \times (\rA_{I_1} \cW) \cong \rA_{I_1} \cW
\ee
by Proposition \ref{absorbing law}. We replace $I$ by $I_1$ and $\cW$ by $\cW \upharpoonright I_1$ so that the resulting $I_0$ becomes $\emptyset$. For any $i \in I$, $\cW(i)$ belongs to either $\rP \cR_{\rank_{\rP \cR}(\cW(i))}$ or $\rM \cR_{< \rank_{\rP \cR}(\cW(i))}$ by $\cW(i) \notin F_k$ and induction hypothesis, and hence $\cW(i)$ is the product of one or two elements of $\rP \cR_{< \alpha}$.

\vs
Therefore, we may replace $I$ and $\cW$ so that the relation $\cW \in \rP \cR_{< \alpha}^I$ holds. By Proposition \ref{successor non-pure}, we obtain $\alpha \in \Suc$ and $V \in \rM \cR_{\alpha_{-}} \subset \rM \cR_{< \alpha}$.
\end{proof}

\section{Reductive Summand}
\label{Reductive Summand}

We study classification of Banach $k$-vector spaces in $\cR$ in terms of pure Reid hierarchy. For this purpose, we introduce a notion which is a little wider than direct summand.

\begin{dfn}
Let $(V,W) \in \Ban(k)^2$. We say that $W$ is a {\it reductive summand of $V$} if there exists a $(\sigma,\pi) \in \Hom_{\leq 1}(W,V) \times \Hom_{\leq 1}(V,W)$ such that $\n{\id_W - \pi \circ \sigma} < 1$. We call such an $(\sigma,\pi)$ a {\it factorisation pair for $(V,W)$}.
\end{dfn}

\begin{prp}
\label{composition}
Let $(V_0,V_1,V_2) \in \Ban(k)^3$. If $V_0$ is a reductive summand of $V_1$ and $V_1$ is a reductive summand of $V_2$, then $V_0$ is a reductive summand of $V_2$.
\end{prp}

\begin{proof}
Take factorisation pairs $(\sigma,\pi)$ and $(\sigma',\pi')$ for $(V_1,V_0)$ and $(V_2,V_1)$ respectively. We have
\be
\pi \circ \pi' \circ \sigma' \circ \sigma = \pi \circ (\id_{V_1} - (\id_{V_1} - \pi' \circ \sigma')) \circ \sigma = \pi \circ \sigma - \pi \circ (\id_{V_1} - \pi' \circ \sigma') \circ \sigma
\ee
and hence
\be
\n{\id_{V_0} - \pi \circ \pi' \circ \sigma' \circ \sigma} & = & \n{\id_{V_0} - \pi \circ \sigma + \pi \circ (\id_{V_1} - \pi' \circ \sigma') \circ \sigma} \\
& \leq & \max \ens{\n{\id_{V_0} - \pi \circ \sigma}, \n{\pi \circ (\id_{V_1} - \pi' \circ \sigma') \circ \sigma}} \\
& \leq & \max \ens{\n{\id_{V_0} - \pi \circ \sigma}, \n{\pi} \ \n{\id_{V_1} - \pi' \circ \sigma'} \ \n{\sigma}} < 1.
\ee
Therefore, $(\sigma' \circ \sigma,\pi \circ \pi')$ is a factorisation pair for $(V_2,V_0)$.
\end{proof}

We say that $\kappa$ {\it satisfies the non-measurability condition} if no $\alpha \in \kappa$ is $\omega_1$-measurable, or equivalently, if no $\alpha \in \kappa$ is measurable. The following is the key application of non-Archimedean Dugas--Zimmermann-Huisgen's extension of Chase's lemma (cf.\ \cite{Mih26} Corollary 3.9):

\begin{prp}
\label{factorisation pair reduction}
Suppose that $\v{k}$ is dense in $\R_{\geq 0}$ and $\kappa$ satisfies the non-measurability condition. Then for any tuple $(I,J,\cW,\cV)$ of an $I \in \Set_{< \kappa}$, a $J \in \Set_{< \kappa}$, a $\cW \in \Ban(k)^I$, a $\cV \in \Ban(k)^J$, the following hold:
\bi
\item[(1)] For any factorisation pair $(\sigma,\pi)$ for $(\Pi_I \cW, \Sigma_J \cV)$, there exists an $(I'_0,J'_0) \in \cP_{< \omega}(I) \times \cP_{< \omega}(J)$ such that for any $(I_0,J_0) \in \cP_{< \omega}(I) \times \cP_{< \omega}(J)$ with $I'_0 \subset I_0$ and $J'_0 \subset J_0$, $(\sigma_{J \setminus J_0}^{I_0},\pi_{I_0}^{J \setminus J_0})$ is a factorisation pair for $(\Pi_{I_0} \cW, \Sigma_{J \setminus J_0} \cV)$.
\item[(2)] For any factorisation pair $(\sigma,\pi)$ for $(\Sigma_J \cV,\Pi_I \cW)$, there exists an $(I'_0,J'_0) \in \cP_{< \omega}(I) \times \cP_{< \omega}(J)$ such that for any $(I_0,J_0) \in \cP_{< \omega}(I) \times \cP_{< \omega}(J)$ with $I'_0 \subset I_0$ and $J'_0 \subset J_0$, $(\sigma_{I \setminus I_0}^{J_0},\pi_{J_0}^{I \setminus I_0})$ is a factorisation pair for $(\Sigma_{J_0} \cV, \Pi_{I \setminus I_0} \cW)$.
\ei
\end{prp}

Although (2) is the same as (1) except for the replacement of $\Pi$ and $\Sigma$ by $\Sigma$ and $\Pi$ respectively, we avoid to use $\rA$ and $\rB$ because proofs are not parallel for the two statements.

\begin{proof}
Set $V \coloneqq \Pi_I \cW$ and $W \coloneqq \Sigma_J \cV$.

\vs
(1) First, suppose $\pi = 0$, then we have $\n{\id_W} = \n{\id_W - \pi \circ \sigma} < 1$ and hence $W = \ens{0}$. Therefore, $(I'_0,J'_0) = (\emptyset,\emptyset)$ satisfies the desired condition.

\vs
Next, suppose $\pi \neq 0$. Since $\v{k}$ is dense in $\R_{\geq 0}$ and $\kappa$ satisfies the non-measurability condition, there exists an $(I'_0,J'_0) \in \cP_{< \omega}(I) \times \cP_{< \omega}(J)$ such that $\n{\pi_{I \setminus I'_0}^{J \setminus J'_0}} < \n{\pi}$ by \cite{Mih26} Corollary 3.9 applied to $\pi$. Let $(I_0,J_0) \in \cP_{< \omega}(I) \times \cP_{< \omega}(J)$ with $I'_0 \subset I_0$ and $J'_0 \subset J_0$. Then we have $\n{\pi_{I \setminus I_0}^{J \setminus J_0}} \leq \n{\pi_{I \setminus I'_0}^{J \setminus J'_0}} < \n{\pi} \leq 1$. Set $V' \coloneqq \Pi_{I_0} \cW$ and $W' \coloneqq \Sigma_{J \setminus J_0} \cV$. We have
\be
\pi_{I_0}^{J \setminus J_0} \circ \sigma_{J \setminus J_0}^{I_0} & = & \pi^{J \setminus J_0} \circ \pi \circ \sigma_{I_0} \circ \pi^{I_0} \circ \sigma \circ \sigma_{J \setminus J_0} \\
& = & \pi^{J \setminus J_0} \circ \pi \circ ( \id_V - \sigma_{I \setminus I_0} \circ \pi^{I \setminus I_0}) \circ \sigma \circ \sigma_{J \setminus J_0} \\
& = & \pi^{J \setminus J_0} \circ \pi \circ \sigma \circ \sigma_{J \setminus J_0} - \pi^{J \setminus J_0} \circ \pi \circ \sigma_{I \setminus I_0} \circ \pi^{I \setminus I_0} \circ \sigma \circ \sigma_{J \setminus J_0} \\
& = & \pi^{J \setminus J_0} \circ (\id_W - (\id_W - \pi \circ \sigma)) \circ \sigma_{J \setminus J_0} - \pi_{I \setminus I_0}^{J \setminus J_0} \circ \sigma_{J \setminus J_0}^{I \setminus I_0} \\
& = & \pi^{J \setminus J_0} \circ \sigma_{J \setminus J_0} - \pi^{J \setminus J_0} \circ (\id_W - \pi \circ \sigma) \circ \sigma_{J \setminus J_0} - \pi_{I \setminus I_0}^{J \setminus J_0} \circ \sigma_{J \setminus J_0}^{I \setminus I_0} \\
& = & \id_{W'} - (\id_W - \pi \circ \sigma)_{J \setminus J_0}^{J \setminus J_0} - \pi_{I \setminus I_0}^{J \setminus J_0} \circ \sigma_{J \setminus J_0}^{I \setminus I_0},
\ee
and hence
\be
\n{\id_{W'} - \pi_{I_0}^{J \setminus J_0} \circ \sigma_{J \setminus J_0}^{I_0}} & = & \n{(\id_W - \pi \circ \sigma)_{J \setminus J_0}^{J \setminus J_0} + \pi_{I \setminus I_0}^{J \setminus J_0} \circ \sigma_{J \setminus J_0}^{I \setminus I_0}} \\
& \leq & \max \ens{\n{(\id_W - \pi \circ \sigma)_{J \setminus J_0}^{J \setminus J_0}}, \n{\pi_{I \setminus I_0}^{J \setminus J_0} \circ \sigma_{J \setminus J_0}^{I \setminus I_0}}} \\
& \leq & \max \ens{\n{\id_W - \pi \circ \sigma}, \n{\pi_{I \setminus I_0}^{J \setminus J_0}} \ \n{\sigma}} < 1.
\ee
Therefore, $(\sigma_{J \setminus J_0}^{I_0},\pi_{I_0}^{J \setminus J_0})$ is a factorisation pair for $(V',W')$.

\vs
(2) First, suppose $\sigma = 0$, then we have $\n{\id_V} = \n{\id_V - \pi \circ \sigma} < 1$ and hence $V = \ens{0}$. Therefore, $(I'_0,J'_0) = (\emptyset,\emptyset)$ satisfies the desired condition.

\vs
Next, suppose $\sigma \neq 0$. Since $\v{k}$ is dense in $\R_{\geq 0}$ and $\kappa$ satisfies the non-measurability condition, there exists an $(I'_0,J'_0) \in \cP_{< \omega}(I) \times \cP_{< \omega}(J)$ such that $\n{\sigma_{I \setminus I'_0}^{J \setminus J'_0}} < \n{\sigma}$ by \cite{Mih26} Corollary 3.9 applied to $\sigma$. Let $(I_0,J_0) \in \cP_{< \omega}(I) \times \cP_{< \omega}(J)$ with $I'_0 \subset I_0$ and $J'_0 \subset J_0$. We have $\n{\sigma_{I \setminus I_0}^{J \setminus J_0}} \leq \n{\sigma_{I \setminus I'_0}^{J \setminus J'_0}} < \n{\sigma} \leq 1$. Set $V' \coloneqq \Pi_{I \setminus I_0} \cW$ and $W' \coloneqq \Sigma_{J_0} \cV$. We have
\be
\pi_{J_0}^{I \setminus I_0} \circ \sigma_{I \setminus I_0}^{J_0} & = & \pi^{I \setminus I_0} \circ \pi \circ \sigma_{J_0} \circ \pi^{J_0} \circ \sigma \circ \sigma_{I \setminus I_0} \\
& = & \pi^{I \setminus I_0} \circ \pi \circ (\id_W - \sigma_{J \setminus J_0} \circ \pi^{J \setminus J_0}) \circ \sigma \circ \sigma_{I \setminus I_0} \\
& = & \pi^{I \setminus I_0} \circ \pi \circ \sigma \circ \sigma_{I \setminus I_0} - \pi^{I \setminus I_0} \circ \pi \circ \sigma_{J \setminus J_0} \circ \pi^{J \setminus J_0} \circ \sigma \circ \sigma_{I \setminus I_0} \\
& = & \pi^{I \setminus I_0} \circ (\id_V - (\id_V - \pi \circ \sigma)) \circ \sigma_{I \setminus I_0} - \pi_{J \setminus J_0}^{I \setminus I_0} \circ \sigma_{I \setminus I_0}^{J \setminus J_0} \\
& = & \pi^{I \setminus I_0} \circ \sigma_{I \setminus I_0} - \pi^{I \setminus I_0} \circ (\id_V - \pi \circ \sigma) \circ \sigma_{I \setminus I_0} - \pi_{J \setminus J_0}^{I \setminus I_0} \circ \sigma_{I \setminus I_0}^{J \setminus J_0} \\
& = & \id_{V'} - (\id_V - \pi \circ \sigma)_{I \setminus I_0}^{I \setminus I_0} - \pi_{J \setminus J_0}^{I \setminus I_0} \circ \sigma_{I \setminus I_0}^{J \setminus J_0},
\ee
and hence
\be
\n{\id_{V'} - \pi_{J_0}^{I \setminus I_0} \circ \sigma_{I \setminus I_0}^{J_0}} & = & \n{(\id_V - \pi \circ \sigma)_{I \setminus I_0}^{I \setminus I_0} + \pi_{J \setminus J_0}^{I \setminus I_0} \circ \sigma_{I \setminus I_0}^{J \setminus J_0}} \\
& \leq & \max \ens{\n{(\id_V - \pi \circ \sigma)_{I \setminus I_0}^{I \setminus I_0}}, \n{\pi_{J \setminus J_0}^{I \setminus I_0} \circ \sigma_{I \setminus I_0}^{J \setminus J_0}}} \\
& \leq & \max \ens{\n{\id_V - \pi \circ \sigma}, \n{\pi} \ \n{\sigma}_{I \setminus I_0}^{J \setminus J_0}} < 1.
\ee
Therefore, $(\sigma_{I \setminus I_0}^{J_0},\pi_{J_0}^{I \setminus I_0})$ is a factorisation pair for $(W',V')$.
\end{proof}

We give a non-Archimedean analogue of \cite{Eda83-2} Lemma 5:

\begin{prp}
\label{pure reductive summand}
Suppose that $\v{k}$ is dense in $\R_{\geq 0}$ and $\kappa$ satisfies the non-measurability condition. Then for any tuple $(\alpha,\beta,V,W)$ of an $\alpha \in \Ord \setminus \ens{0}$, a $\beta \in \Ord \setminus \ens{0}$, a $V \in \rP \Pi_{\alpha}$, and a $W \in \rP \Sigma_{\beta}$, the following hold:
\bi
\item[(1)] If $W$ is a reductive summand of $V$, then the inequality $\beta < \alpha$ holds.
\item[(2)] If $V$ is a reductive summand of $W$, then the inequality $\alpha < \beta$ holds.
\ei
\end{prp}

\begin{proof}
We show the logical conjunction of (1) and (2) by transfinite induction on $\min \ens{\alpha,\beta}$. By $V \in \rP \Pi_{\alpha}$, there exists a pair $(I,\cW)$ of an $I \in \Set_{< \kappa}$ and a $\cW \in \rP \cR_{< \alpha}^{I \to \Sigma}$ with $V \cong \Pi_I \cW$. By $W \in \rP \Sigma_{\alpha}$, there exists a pair $(J,\cV)$ of a $J \in \Set_{< \kappa}$ and a $\cV \in \rP \cR_{< \beta}^{J \to \Pi}$ with $W \cong \Sigma_J \cV$. It is reduced to the case $V = \Pi_I \cW$ and $W = \Sigma_J \cV$. Set $I_{\rP \Pi} \coloneqq \set{i \in I}{\cW(i) \notin \rP \Sigma_{< \alpha}}$, $I_{\rP \Sigma} \coloneqq I \setminus I_{\rP \Pi}$, $J_{\rP \Sigma} \coloneqq \set{j \in J}{\cV(j) \notin \rP \Pi_{< \alpha}}$, $J_{\rP \Pi} \coloneqq J \setminus J_{\rP \Sigma}$.

\vs
Suppose $\alpha \in \Suc$. Set $r_0 \coloneqq \rank_{\rP \Pi}(\Pi_{I_{\rP \Pi}} \cW)$. By Proposition \ref{successor non-pure}, we have $\Pi_{I_{\rP \Pi}} \cW \in \rP \Pi_{< \alpha}$, and hence $r_0 \leq \alpha_{-}$. Take a pair $(I',\cW')$ of a $I' \in \Set_{< \kappa}$ and a $\cW' \in \rP \cR_{< r_0}^{I' \to \Sigma}$ with $\Pi_{I_{\rP \Pi}} \cW \cong \Pi_{I'} \cW'$. We replace $I$ by $I' \sqcup I_{\rP \Sigma}$ and $\cW$ by the map $I' \sqcup I_{\rP \Sigma} \to \rP \cR_{< \alpha}$ defined by
\be
i \mapsto 
\left\{
\begin{array}{ll}
\cW'(i) & (i \in I') \\
\cW(i) & (i \in I_{\rP \Sigma})
\end{array}
\right.
\ee
so that the relation $\cW(i) \notin \rP \Pi_{\alpha_{-}}$ holds for any $i \in I$ by $r_0 \leq \alpha_{-}$.

\vs
By the argument above, we may assume that the relation $\alpha \notin \Suc$ holds or there exists an $i \in I_{\rP \Sigma}$ such that for any $i' \in I_{\rP \Pi}$, the inequality $\rank_{\rP \Pi}(\cW(i')) < \rank_{\rP \Sigma}(\cW(i))$ holds. In particular, we may assume that for any $I_0 \in \cP_{< \omega}(I)$, there exists an $i \in I_{\rP \Sigma}$ such that for any $i' \in I_0 \cap I_{\rP \Pi}$, the inequality $\rank_{\rP \Pi}(\cW(i')) < \rank_{\rP \Sigma}(\cW(i))$ holds. Similarly, we may assume that for any $J_0 \in \cP_{< \omega}(J)$, there exists a $j \in J_{\rP \Pi}$ such that for any $j' \in J_0 \cap J_{\rP \Sigma}$, the inequality $\rank_{\rP \Sigma}(\cV(j')) < \rank_{\rP \Pi}(\cV(j))$ holds.

\vs
(1) Take a factorisation pair $(\sigma,\pi)$ for $(V,W)$. By Proposition \ref{factorisation pair reduction} (1), there exists an $(I_0,J_0) \in \cP_{< \omega}(I) \times \cP_{< \omega}(J)$ such that $(\sigma_{J \setminus J'_0}^{I'_0},\pi_{I'_0}^{J \setminus J'_0})$ is a factorisation pair for $(\Pi_{I'_0} \cW, \Sigma_{J \setminus J'_0} \cV)$ for any $(I'_0,J'_0) \in \cP_{< \omega}(I) \times \cP_{< \omega}(J)$ with $I_0 \subset I'_0$ and $J_0 \subset J'_0$. 

\vs
Set $V' \coloneqq \Pi_{I_0} \cW$ and $W' \coloneqq \Sigma_{J \setminus J_0} \cV$. Inserting an element of $I_{\rP \Sigma}$ to $I_0$, we may assume that there exists an $i \in I_0 \cap I_{\rP \Sigma}$ such that for any $i' \in I_0 \cap I_{\rP \Pi}$, the inequality $\rank_{\rP \Pi}(\cW(i')) < \rank_{\rP \Sigma}(\cW(i))$ holds. Then, we have $V' \in \rP \Sigma_{< \alpha}$ by Proposition \ref{pure finite product} and Proposition \ref{mixed finite product}.

\vs
Set $\alpha' \coloneqq \rank_{\rP \Sigma}(V')$. Assume $\beta \geq \alpha$. By $\cV \in \rP \cR_{< \beta}^{J \to \Pi}$, $\# J_0 < \omega$, and $\alpha' < \alpha \leq \beta$, there exists a $j \in J \setminus J_0$ such that $\cV(j) \in \rP \Sigma_{< \beta} \setminus \rP \Sigma_{< \alpha'}$. Set $\beta' \coloneqq \rank_{\rP \Sigma}(\cV(j))$. By $\alpha' \leq \beta' < \beta$ and $\alpha' < \alpha$, we have $\min \ens{\alpha',\beta'} = \alpha' < \min \ens{\alpha,\beta}$. Since $\cV(j)$ is a direct summand of $W'$ and $W'$ is a reductive summand of $V'$, $\cV(j)$ is a reductive summand of $V'$ by Proposition \ref{composition}. This contradicts induction hypothesis by $V' \in \rP \Sigma_{\alpha'}$ and $\cV(j) \in \rP \Pi_{\beta'}$. Therefore, we obtain $\beta < \alpha$.

\vs
(2) Take a factorisation pair $(\sigma,\pi)$ for $(W,V)$. By Proposition \ref{factorisation pair reduction} (2), there exists an $(I_0,J_0) \in \cP_{< \omega}(I) \times \cP_{< \omega}(J)$ such that $(\sigma_{I \setminus I'_0}^{J'_0},\pi_{J'_0}^{I \setminus I'_0})$ is a factorisation pair for $(\Sigma_{J'_0} \cV,\Pi_{I \setminus I'_0} \cW)$ for any $(I'_0,J'_0) \in \cP_{< \omega}(I) \times \cP_{< \omega}(J)$ with $I_0 \subset I'_0$ and $J_0 \subset J'_0$.

\vs
Set $V' \coloneqq \Pi_{I \setminus I_0} \cW$ and $W' \coloneqq \Sigma_{J_0} \cV$. Inserting an element of $J_{\rP \Pi}$ to $J_0$, we may assume that there exists a $j \in J_0 \cap J_{\rP \Pi}$ such that for any $j' \in J_0 \cap J_{\rP \Sigma}$, the inequality $\rank_{\rP \Sigma}(\cV(j')) < \rank_{\rP \Pi}(\cV(j))$ holds. Then, we have $W' \in \rP \Pi_{< \alpha}$ by Proposition \ref{pure finite product} and Proposition \ref{mixed finite product}.

\vs
Set $\beta' \coloneqq \rank_{\rP \Pi}(W')$. Assume $\alpha \geq \beta$. By $\cW \in \rP \cR_{< \alpha}^{I \to \Sigma}$, $\# I_0 < \omega$, and $\beta' < \beta \leq \alpha$, there exists an $i \in I \setminus I_0$ such that $\cW(i) \in \rP \Pi_{< \alpha} \setminus \rP \Pi_{< \beta'}$. Set $\alpha' \coloneqq \rank_{\rP \Pi}(\cW(i))$. By $\beta' \leq \alpha' < \alpha$ and $\beta' < \beta$, we have $\min \ens{\alpha',\beta'} = \beta' < \min \ens{\alpha,\beta}$. Since $\cW(i)$ is a direct summand of $V'$ and $V'$ is a reductive summand of $W'$, $\cW(i)$ is a reductive summand of $W'$ by Proposition \ref{composition}. This contradicts induction hypothesis by $W' \in \rP \Pi_{\beta'}$ and $\cW(i) \in \rP \Sigma_{\alpha'}$. Therefore, we obtain $\alpha < \beta$.
\end{proof}

\begin{crl}
\label{pure disjointness 1}
Suppose that $\v{k}$ is dense in $\R_{\geq 0}$ and $\kappa$ satisfies the non-measurability condition. Then for any $(\alpha,\beta) \in (\Ord \setminus \ens{0})^2$, the equality $\rP \Pi_{\alpha} \cap \rP \Sigma_{\beta} = \emptyset$ holds.
\end{crl}

\begin{proof}
Assume that $\rP \Pi_{\alpha} \cap \rP \Sigma_{\beta}$ has an element $V$. Then $V$ is a reductive summand of $V$ itself, but this contradicts Proposition \ref{pure reductive summand}.
\end{proof}

\begin{crl}
\label{pure disjointness 2}
Suppose that $\v{k}$ is dense in $\R_{\geq 0}$ and $\kappa$ satisfies the non-measurability condition. Then for any $(\alpha,\beta) \in \Ord^2$ with $\alpha \neq \beta$, the equality $\rP \Pi_{\alpha} \cap \rP \Pi_{\beta} = \rP \Sigma_{\alpha} \cap \rP \Sigma_{\beta} = \emptyset$ holds.
\end{crl}

\begin{proof}
It is reduced to the case $\beta < \alpha$ by symmetry. Assume that $\rP \rA_{\alpha} \cap \rP \rA_{\beta}$ has an element $V$. By $\alpha > \beta \geq 0$, there exists a pair $(I,\cW)$ of an $I \in \Set_{< \kappa}$ and a $\cW \in \rP \cR_{< \alpha}^{I \to \rB}$ with $V \cong \rA_I \cW$. By $\cW \in \rP \cR_{< \alpha}^{I \to \rB}$, there exists an $i \in I$ with $\cW(i) \in \rP \rB_{< \alpha} \setminus \rP \rB_{< \beta}$. In particular, we have $\rank_{\rP \rB}(\cW(i)) \geq \beta$. Since $\cW(i)$ is a reductive summand of $V$, we have $\beta = 0$ by Proposition \ref{pure reductive summand}. This contradicts $\alpha \neq 0$ and Proposition \ref{pure rank 0}.
\end{proof}

\begin{crl}
\label{mixed reduction}
Suppose that $\v{k}$ is dense in $\R_{\geq 0}$ and $\kappa$ satisfies the non-measurability condition. Then for any $(\alpha,\beta) \in (\Ord \setminus \ens{0})^2$, if there exists a $(V,W) \in \rM \cR_{\alpha} \times \rP \cR_{\beta}$ such that $W$ is a reductive summand of $V$, then the inequality $\beta \leq \alpha$ holds.
\end{crl}

\begin{proof}
Assume $W \in \rP \rB_{\beta}$. Take a $(V_{\rA},V_{\rB}) \in \rP \rA_{\alpha} \times \rP \rB_{\alpha}$ with $V \cong V_{\Pi} \times V_{\Sigma}$. By $\alpha \neq 0$, there exists a pair $(I,\cW)$ of an $I \in \Set_{< \kappa}$ and a $\cW \in \rP \cR_{< \alpha}^{I \to \rB}$ with $V_{\rA} \cong \rA_I \cW$. By $\beta \neq 0$, there exists a pair $(J,\cV)$ of a $J \in \Set_{< \kappa}$ and a $\cV \in \rP \cR_{< \beta}^{J \to \rA}$ with $W \cong \rB_J \cV$.

\vs
Since $W \cong \rB_J \cV$ is a reductive summand of $V \cong (\rA_I \cW) \times V_{\rB}$, there exists an $(I_0,J_0) \in \cP_{< \omega}(I) \times \cP_{< \omega}(J)$ such that $\rB_{J \setminus J_0} \cV$ is a reductive summand of $(\rA_{I_0} \cW) \times V_{\rB}$ by Proposition \ref{factorisation pair reduction}. Set $V' \coloneqq \rB_{J \setminus J_0} \cV$ and $W' \coloneqq (\rA_{I_0} \cW) \times V_{\rB}$. By Proposition \ref{pure finite product} and Proposition \ref{mixed finite product}, we have $W' \in \rP \Sigma_{\alpha}$.

\vs
We show $\cV(j) \in \rP \rA_{< \alpha}$ for any $j \in J \setminus J_0$. Since $\cV(j)$ is a direct summand of $V'$ and $V'$ is a reductive summand of $W'$, $\cV(j)$ is a reductive summand of $W'$ by Proposition \ref{composition}. Therefore, we have $\cV(j) \in \rP \rA_{< \alpha}$ by Proposition \ref{pure reductive summand}. We obtain
\be
\# \set{j \in J}{\cV(j) \in \rP \rA_{< \beta} \setminus \rP \rA_{< \alpha}} \leq \# \set{j \in J}{\cV(j) \notin \rP \rA_{< \alpha}} \leq \# J_0 < \omega.
\ee
This implies $\beta \leq \alpha$ by $\cV \in \rP \cR_{< \beta}^{J \to \rA}$.
\end{proof}

\begin{crl}
\label{pure mixed disjointness}
Suppose that $\v{k}$ is dense in $\R_{\geq 0}$ and $\kappa$ satisfies the non-measurability condition. Then for any $(\alpha,\beta) \in (\Ord \setminus \ens{0})^2$, the equality $\rM \cR_{\alpha} \cap \rP \cR_{\beta} = \emptyset$ holds.
\end{crl}

\begin{proof}
Assume that $\rM \cR_{\alpha} \cap \rP \rA_{\beta}$ has an element $V$. Take a $(V_{\Pi},V_{\Sigma}) \in \rP \Pi_{\alpha} \times \rP \Sigma_{\alpha}$ with $V \cong V_{\Pi} \times V_{\Sigma}$.

\vs
Since $V \in \rP \rA_{\beta}$ is a direct summand of $V \in \rM \cR_{\alpha}$ itself, we have $\beta < \alpha$ by Corollary \ref{mixed reduction}. Since $V_{\rB} \in \rP \Pi_{\alpha}$ is a reductive summand of $V \in \rP \rA_{\beta}$, we have $\alpha < \beta$ by Corollary \ref{pure reductive summand}. This leads to contradiction.
\end{proof}

\begin{crl}
\label{mixed disjointness}
Suppose that $\v{k}$ is dense in $\R_{\geq 0}$ and $\kappa$ satisfies the non-measurability condition. Then for any $(\alpha,\beta) \in \Ord^2$ with $\alpha \neq \beta$, the equality $\rM \cR_{\alpha} \cap \rM \cR_{\beta} = \emptyset$ holds.
\end{crl}

\begin{proof}
It is reduced to the case $\beta < \alpha$ by symmetry. Assume that $\rM \cR_{\alpha} \cap \rM \cR_{\beta}$ has an element $V$. We have $\beta \neq 0$ by Proposition \ref{pure rank 0} (2) and Proposition \ref{degeneracy of mixed}. Take a $(V_{\Pi},V_{\Sigma}) \in \rP \Pi_{\alpha} \times \rP \Sigma_{\alpha}$ with $V \cong V_{\Pi} \times V_{\Sigma}$. Since $V_{\Pi} \in \rP \Pi_{\alpha}$ is a reductive summand of $V \in \rM \cR_{\beta}$, we have $\alpha \leq \beta$ by Corollary \ref{mixed reduction}, which contradicts $\beta < \alpha$.
\end{proof}

We obtain a non-Archimedean analogue of \cite{Eda83-2} Theorem 1:

\begin{thm}
\label{non-Archimedean Eda's theorem}
Suppose that $\v{k}$ is dense in $\R_{\geq 0}$ and $\kappa$ satisfies the non-measurability condition. Then $\cR$ is the disjoint union of $F_k$, $\rM \cR_{\alpha}$, $\rP \Pi_{\alpha}$, and $\rP \Sigma_{\alpha}$ with $\alpha \in \Ord \setminus \ens{0}$.
\end{thm}

We note that Theorem \ref{non-Archimedean Eda's theorem} does not claim the non-emptiness of each group. Indeed, a group can be empty by Proposition \ref{essentially small}.

\begin{proof}
By Proposition \ref{pure rank 0} (1), Proposition \ref{degeneracy of mixed} (2), and Theorem \ref{mixed decomposition}, we have
\be
\cR = F_k \cup \bigcup_{\alpha \in \Ord \setminus \ens{0}} \rM \cR_{\alpha} \cup \bigcup_{\alpha \in \Ord \setminus \ens{0}} \rP \Pi_{\alpha} \cup \bigcup_{\alpha \in \Ord \setminus \ens{0}} \rP \Sigma_{\alpha}.
\ee
The disjointedness of $F_k$ and the other components follows from Proposition \ref{pure rank 0} (2) and Proposition \ref{degeneracy of mixed} (2). The disjointedness of $M_k$ and the other components follows from Corollary \ref{pure mixed disjointness} and Corollary \ref{mixed disjointness}. The disjointedness for the last two groups of components follows from Corollary \ref{pure disjointness 1} and Corollary \ref{pure disjointness 2}.
\end{proof}

\begin{crl}
\label{Delta characterisation}
Suppose that $\v{k}$ is dense in $\R_{\geq 0}$ and $\kappa$ satisfies the non-measurability condition. Then for any $\alpha \in \Ord$, the following hold:
\bi
\item[(1)] If $\alpha = 0$, then the equality $\Delta_0 = F_k$ holds.
\item[(2)] If $\alpha \in \Suc$, then the equality $\Delta_{\alpha} = \cR_{< \alpha} \cup \rM \cR_{\alpha_{-}}$ holds.
\item[(3)] If $\alpha \notin \Suc \cup \ens{0}$, then the equality $\Delta_{\alpha} = \cR_{< \alpha}$ holds.
\ei
\end{crl}

\begin{proof}
The assertions immediately follow from Proposition \ref{rank 0}, Proposition \ref{degeneracy of mixed}, and Theorem \ref{non-Archimedean Eda's theorem}.
\end{proof}

We obtain a non-Archimedean analogue of Zimmermann-Huisgen's theorem distinguishing various $\Z$-kernel groups (cf.\ \cite{Zim79}).

\begin{crl}
\label{distinguishing}
Suppose that $\v{k}$ is dense in $\R_{\geq 0}$ and $\kappa$ satisfies the non-measurability condition. Let $\lambda \in (\Card \cap \kappa) \setminus \omega$. For each $(t,n) \in \ens{0,1} \times \omega$, we define $V_{t,n}$ in the following recursive way:
\be
V_{t,n} & \coloneqq &
\left\{
\begin{array}{ll}
k & (n = 0) \\
\ell^{\infty}(\lambda,V_{1-t,n-1}) & (n \neq 0 \land t = 0) \\
\rC_0(\lambda,V_{1-t,n-1}) & (n \neq 0 \land t = 1)
\end{array}
\right.
\ee
Then for any $((t,n),(t',n')) \in (\ens{0,1} \times \omega_{> 0})^2$, the following are equivalent:
\bi
\item[(1)] The isomorphism $V_{t,n} \cong V_{t',n'}$ holds.
\item[(2)] The equality $(t,n) = (t',n')$ holds.
\ei
\end{crl}

\begin{proof}
For any $(t,n) \in \ens{0,1} \times \omega_{> 0}$, we have $V_{t,n} \in \rP \Pi_n$ if $t = 0$ and $V_{t,n} \in \rP \Sigma_n$ if $t = 1$ by induction on $n$. Therefore, the assertion follows from Theorem \ref{non-Archimedean Eda's theorem}.
\end{proof}

By Corollary \ref{distinguishing} applied to $\lambda = \omega$ (with replacement of $\kappa$ by $\omega_1$ if $\kappa$ is $\omega$), we obtain that the isomorphism classes of the Banach $k$-vector spaces
\be
& & \ell^{\infty}(\omega,k), \rC_0(\omega,k), \ell^{\infty}(\omega,\rC_0(\omega,k)), \rC_0(\omega,\ell^{\infty}(\omega,k)), \\
& & \ell^{\infty}(\omega,\rC_0(\omega,\ell^{\infty}(\omega,k))), \rC_0(\omega,\ell^{\infty}(\omega,\rC_0(\omega,k))),
\ee
and so on are all distinct if $\v{k}$ is dense in $\R_{\geq 0}$.

\section{Counterexmples}
\label{Counterexamples}

In the following in this paper, we assume that $\v{k}$ is dense in $\R_{\geq 0}$ and $\kappa$ satisfies the non-measurability condition. We provides examples of Banach $k$-vector spaces which do not belong to $\cR$.

\begin{thm}
\label{bounded continuous function}
The Banach $k$-vector space $\rCbd(\Q,k)$ of bounded continuous functions $\Q \to k$ equipped with the supremum norm does not belong to $\cR$.
\end{thm}

In order to prove Theorem \ref{bounded continuous function}, we prepare a non-Archimedean analogue of \cite{Eda83-2} Corollary.

\begin{lmm}
\label{ell infty vs c0}
For any $V \in \cR$, $\ell^{\infty}(\omega,V)$ and $\rC_0(\omega,V)$ are not isomorphic to each other.
\end{lmm}

\begin{proof}
We replace $\kappa$ by $\omega_1$ if $\kappa$ is $\omega$. The assertion immediately follows from Theorem \ref{non-Archimedean Eda's theorem}, because they do not share a group in any case. Indeed, $\ell^{\infty}(\omega,\cdot)$ defines function formulae
\be
F_k & \to & \rP \Pi_1 \\
\rP \cR_{\alpha} \cup \rM \cR_{\alpha} & \to & \rP \Pi_{\alpha + 1}
\ee
and
$\rC_0(\omega,\cdot)$ defines function formulae
\be
F_k & \to & \rP \Sigma_1 \\
\rP \cR_{\alpha} \cup \rM \cR_{\alpha} & \to & \rP \Sigma_{\alpha + 1}
\ee
for any $\alpha \in \Ord \setminus \ens{0}$.
\end{proof}

\begin{proof}[Proof of Theorem \ref{bounded continuous function}]
The proof is completely parallel to the proof of \cite{Eda83-2}. Since $\Q$ is homeomorphic to $\omega \times \Q$, we have
\be
\rCbd(\Q,k) \cong \rCbd(\omega \times \Q,k) \cong \ell^{\infty}(\omega,\rCbd(\Q,k)).
\ee
Since $\Q$ is homeomorphic to the quotient $((\omega + 1) \times \Q)/(\ens{\omega} \times \Q)$ of $(\omega + 1) \times \Q$ under the identifications $(\omega,q_0) \sim (\omega,q_1)$ for all $(q_0,q_1) \in \Q^2$, we have
\be
\rCbd(\Q,k) & \cong & \rCbd(((\omega + 1) \times \Q)/(\ens{\omega} \times \Q),k) \\
& \cong & k^{\ens{\ens{\omega} \times \Q}} \oplus \set{f \in \rCbd(((\omega + 1) \times \Q)/(\ens{\omega} \times \Q),k)}{f(\ens{\omega} \times \Q) = 0} \\
& \cong & k \oplus \set{f \in \rCbd((\omega + 1) \times \Q,k)}{f \upharpoonright (\ens{\omega} \times \Q) = 0} \\
& \cong & k \oplus \rC_0(\omega,\rCbd(\Q,k)).
\ee
If $\kappa = \omega$, then we have $\cR = F_k$, and hence $\rCbd(\Q,k)^{\vee} \notin \cR$. Suppose $\kappa \neq \omega$. If $\rCbd(\Q,k) \in \cR$, then we have
\be
\ell^{\infty}(\omega,\rCbd(\Q,k)) \cong \rCbd(\Q,k) \cong k \oplus \rC_0(\omega,\rCbd(\Q,k)) \cong \rC_0(\omega,\rCbd(\Q,k)),
\ee
by Proposition \ref{absorbing law}, but this contradicts Lemma \ref{ell infty vs c0}.
\end{proof}

We show a non-spherically complete analogue of \cite{Eda83-2} Theorem 2:

\begin{thm}
\label{measure space}
If $k$ is not spherically complete, then $\rCbd(\Q,k)^{\vee}$ does not belong to $\cR$.
\end{thm}

\begin{proof}
By the duality between bounded direct product and completed direct sum (cf.\ \cite{Roo78} 4.21 and \cite{MN89} \S 12 Corollary 7.18), we have
\be
\rCbd(\Q,k)^{\vee} \cong \ell^{\infty}(\omega,\rCbd(\Q,k))^{\vee} \cong \rC_0(\omega,\rCbd(\Q,k)^{\vee})
\ee
and
\be
\rCbd(\Q,k)^{\vee} \cong (k \oplus \rC_0(\omega,\rCbd(\Q,k))^{\vee} \cong k \times \ell^{\infty}(\omega,\rCbd(\Q,k)^{\vee}).
\ee
If $\kappa = \omega$, then we have $\cR = F_k$, and hence $\rCbd(\Q,k)^{\vee} \notin \cR$. Suppose $\kappa \neq \omega$. If $\rCbd(\Q,k)^{\vee} \in \cR$, then we have
\be
\ell^{\infty}(\omega,\rCbd(\Q,k)^{\vee}) \cong \rCbd(\Q,k)^{\vee} \cong k \times \rC_0(\omega,\rCbd(\Q,k)^{\vee}) \cong \rC_0(\omega,\rCbd(\Q,k)^{\vee})
\ee
by Proposition \ref{absorbing law}, but this contradicts Lemma \ref{ell infty vs c0}.
\end{proof}

We denote by $\Rfl \subset \Ban(k)$ the subclass of reflexive Banach $k$-vector spaces.

\begin{prp}
\label{reflexive}
If $k$ is not spherically complete, then the inclusion $\cR \subset \Rfl$ holds.
\end{prp}

\begin{proof}
The assertion immediately follows from $k \in \Rfl$ and the duality between bounded direct product and completed direct sum (cf.\ \cite{Roo78} 4.21 and \cite{MN89} \S 12 Corollary 7.18).
\end{proof}

\begin{crl}
Suppose that $\kappa$ is not $\omega$. If $k$ is not spherically complete, then $\cR$ is not closed under taking quotient by closed subspaces.
\end{crl}

\begin{proof}
By Proposition \ref{reflexive}, the assertion follows from the fact that $\ell^{\infty}(\omega,k)/\rC_0(\omega,k)$ is not reflexive (cf.\ \cite{MN89} \S 12 Theorem 7.14).
\end{proof}

\begin{crl}
Suppose that $\kappa$ is not $\omega$. If $k$ is not spherically complete, then $\cR$ is not closed under taking closed subspaces.
\end{crl}

\begin{proof}
By Proposition \ref{reflexive}, the assertion follows from the fact that $\ell^{\infty}(\omega,k)$ admits a non-reflexive closed subspace (cf.\ \cite{Roo78} 4.J and \cite{PG95} Theorem 2.3 (i) and (iii)).
\end{proof}

\begin{crl}
Suppose that $\kappa$ is not $\omega$. If $k$ is not spherically complete, then $\cR$ is not closed under taking tensor product.
\end{crl}

\begin{proof}
By Proposition \ref{reflexive}, the assertion follows from the fact that $\ell^{\infty}(\omega,k) \hat{\otimes} \ell^{\infty}(\omega,k)$ is not reflexive (cf.\ \cite{PG95} Theorem 2.3 (iii)).
\end{proof}

\begin{thm}
\label{left adjoint}
The forgetful functor $\cR \hookrightarrow \Ban(k)$ does not admit a left adjoint functor.
\end{thm}

A subclass $\cC \subset \Rfl$ is said to be reflexive if $k \in \cC$ and $V^{\vee} \in \cC$ for any $V \in \cC$. In order to prove Theorem \ref{left adjoint}, we prepare several lemmata on a reflexive subclass of $\Rfl$.

\begin{lmm}
\label{universality}
Let $\cC$ be a reflexive subclass of $\Rfl$. Let $V \in \Ban(k)$, $V' \in \cC$, and $\iota \in \Hom_{\leq 1}(V,V')$. If $\iota$ is universal in the sense that the map $\Hom_{\leq 1}(V',W) \to \Hom_{\leq 1}(V,W)$ defined by
\be
f \mapsto f \circ \iota
\ee
is bijective for any $W \in \cC$, then $V'$ is isomorphic to $V^{\vee \vee}$, and $(V')^{\vee}$ is isomorphic to $V^{\vee}$.
\end{lmm}

\begin{proof}
The dual $\Hom(V',k) \to \Hom(V,k)$ of $\iota$ bijectively maps $\Hom_{\leq 1}(V',k)$ to $\Hom_{\leq 1}(V,k)$, and hence is an isomorphism because $\v{k}$ is dense in $\R_{\geq 0}$. Therefore, the second dual
\be
V' \cong (V')^{\vee \vee} = \Hom(V',k)^{\vee} \to \Hom(V,k)^{\vee} = V^{\vee \vee},
\ee
of $\iota$ is also an isomorphism. In particular, $V^{\vee \vee}$ is reflexive, and hence so is $V^{\vee}$ by \cite{Roo78} 4.25. This implies
\be
(V')^{\vee} \cong (V^{\vee \vee})^{\vee} = (V^{\vee})^{\vee \vee} \cong V^{\vee}.
\ee
\end{proof}

\begin{proof}[Proof of Theorem \ref{left adjoint}]
Assume the existence of a left adjoint functor $F$. For any $V \in \Ban(k)$, the adjunction $V \to F(V)$ satisfies the universality in the sense of Lemma \ref{universality}, and hence $V^{\vee}$ belongs to $\cC$. This contradicts $\rC_0(\omega,k)^{\vee} \cong \ell^{\infty}(\omega,k) \notin F_k$ if $\kappa$ is $\omega$ and Theorem \ref{measure space} if $\kappa$ is not $\omega$.
\end{proof}

\vspace{0.3in}
\addcontentsline{toc}{section}{Acknowledgements}
\noindent {\Large \bf Acknowledgements}
\vspace{0.2in}

\noindent
I thank K.\ Eda for introducing to me the preceding study of Reid class. I thank all people who helped me to learn mathematics and programming. I also thank my family.

%

\end{document}